\documentclass[12pt, reqno]{amsart}
\usepackage{amsmath, amsthm, amscd, amsfonts, amssymb, graphicx, color}
\usepackage[bookmarksnumbered, colorlinks, plainpages]{hyperref}
\hypersetup{colorlinks=true,linkcolor=red, anchorcolor=green, citecolor=cyan, urlcolor=red, filecolor=magenta, pdftoolbar=true}

\newtheorem{theorem}{Theorem}[section]
\newtheorem{lemma}[theorem]{Lemma}

\newtheorem{corollary}[theorem]{Corollary}
\theoremstyle{definition}

\theoremstyle{remark}

\numberwithin{equation}{section}

\begin{document}
\setcounter{page}{1}

\title[Jordan derivations on some algebras]{Characterizations of Jordan derivations on algebras of locally measurable operators}

\author[G. An]{Guangyu An}

\address{Department of Mathematics, Shaanxi University of Science and Technology,
 Xi'an 710021, China.}
\email{\textcolor[rgb]{0.00,0.00,0.84}{anguangyu310@163.com}}

\author[J. He]{Jun He}
\address{Department of Mathematics and Physics, Anhui Polytechnic University,
Wuhu 241000, China.}
\email{\textcolor[rgb]{0.00,0.00,0.84}{15121034934@163.com}}



\date{Received: xxxxxx; Revised: yyyyyy; Accepted: zzzzzz.
\newline \indent $^{*}$ Corresponding author}

\begin{abstract}
We prove that if $\mathcal M$ is a properly
infinite von Neumann algebra and $LS(\mathcal M)$ is the local measurable operator algebra
affiliated with $\mathcal M$,
then every Jordan derivation from $LS(\mathcal M)$ into itself is continuous with respect to the local measure topology $t(\mathcal M)$.
We construct an extension of a Jordan derivation
from $\mathcal M$ into $LS(\mathcal M)$ up to a Jordan
derivation from $LS(\mathcal M)$ into itself. Moreover, we prove that
if $\mathcal M$ is a properly von Neumann algebra
and $\mathcal A$ is a subalgebra of $LS(\mathcal M)$
such that $\mathcal M\subset\mathcal A$, then every Jordan derivation
from $\mathcal A$ into $LS(\mathcal M)$ is continuous with respect to the local
measure topology $t(\mathcal M)$.
\end{abstract}\maketitle

\section{Introduction}

Let $\mathcal{R}$ be an associative ring.
For an integer $n\geq2$, $\mathcal{R}$ is said to be \emph{$n$-torsion-free}
if $na=0$ implies that $a=0$ for every $a$ in $\mathcal{R}$. Recall that a ring $\mathcal{R}$
is \emph{prime} if $a\mathcal{R}b=(0)$ implies that either $a=0$ or $b=0$ for each $a$ and $b$ in $\mathcal{R}$; and is \emph{semiprime} if $a\mathcal{R}a=(0)$ implies that
$a=0$ for every $a$ in $\mathcal{R}$.
For each $a$ and $b$ in $\mathcal{R}$, we denote by $[a,b]=ab-ba$.

Suppose that $\mathcal{M}$ is a $\mathcal{R}$-bimodule.
An additive mapping $\delta$ from $\mathcal{R}$ into $\mathcal{M}$ is
called a \emph{derivation} if $\delta(ab)=\delta(a)b+a\delta(b)$ for each $a$ and $b$ in $\mathcal{R}$;
$\delta$ is called a \emph{Jordan derivation} if $\delta(a^{2})=\delta(a)a+a\delta(a)$
for every $a$ in $\mathcal{R}$; and $\delta$ is called an \emph{inner derivation} if there
exists an element $m$ in $\mathcal{M}$, such that
$\delta(a)=[m,a]$ for every $a$ in $\mathcal{R}$.
Obviously, every derivation is a Jordan derivation
and every inner derivation is a derivation. The converse is, in general, not true.

A classical result of Herstein \cite{I. Herstein} proves that every Jordan derivation
on a 2-torsion-free prime ring is a derivation; in \cite{M. Bresar J. Vukman1},
Bre\v{s}ar and Vukman give a brief proof of \cite[Theorem 3.1]{I. Herstein}.
In \cite{Cusack}, Cusack generalizes \cite[Theorem 3.1]{I. Herstein} to
2-torsion-free semiprime rings; in \cite{M. Bresar2}, Bre\v{s}ar gives an
alternative proof of \cite[Corollary 5]{Cusack}.
Let $\mathcal A$ be a $C^*$-algebra and $\mathcal M$ be a Banach $\mathcal A$-bimodule,
in \cite{Ringrose}, Ringrose shows that every derivation from $\mathcal A$ into $\mathcal M$ is continuous;
in \cite{hejazian}, Hejazian and Niknam prove that if
$\mathcal A$ is commutative, then every Jordan derivation from $\mathcal A$ into $\mathcal M$ is continuous;
in \cite{Johnson}, Johnson shows that every continuous Jordan derivation
from $\mathcal A$ into $\mathcal M$ is a derivation;
by Cunts \cite[Corollary 1.2]{Cuntz} and Johnson \cite[Theorem 6.3]{Johnson},
we know that every Jordan derivation from $\mathcal A$ into $\mathcal M$ is a derivation.
In \cite{sakai}, Sakai proves that every derivation on a von Neumann algebra is an inner derivation.

Compare with the characterizations of derivations on Banach algebras, investigation of
derivations on  unbounded operator algebras begin much later.

In \cite{Segal}, Segal studies the theory of noncommutative integration,
and introduces various classes of non-trivial $*$-algebras of unbounded
operators. In this paper, we mainly consider
the $*$-algebra $S(\mathcal M)$ of all measurable operators
and the $*$-algebra $LS(\mathcal M)$ of all locally measurable operators
affiliated with a von Neumann algebra $\mathcal M$.
In \cite{Segal}, Segal shows that the algebraic and topological properties
of the measurable operators algebra $S(\mathcal M)$ are similar to the von Neumann algebra $\mathcal M$.
If $\mathcal M$ is a commutative von Neumann algebra,
then $\mathcal M$ is $*$-isomorphic to the algebra
$L^\infty(\Omega,\Sigma,\mu)$ of all essentially bounded measurable
complex functions on a measure space $(\Omega,\Sigma,\mu)$; and
$S(\mathcal M)$ is $*$-isomorphic to the algebra $L^0(\Omega,\Sigma,\mu)$
of all measurable almost everywhere finite complex-valued functions on
$(\Omega,\Sigma,\mu)$. In \cite{Ber1}, Ber, Chilin and Sukochev show that there exists a derivation on $L^0(0,1)$
is not an inner derivation, and the derivation is discontinuous in the measure topology.
This result means that the properties of derivations on $S(\mathcal M)$
are different from the derivations
on $\mathcal M$.

In \cite{Albeverio1, Albeverio2}, Albeverio, Ayupov and Kudaybergenov study
the properties of derivations on various classes of measurable algebras.
If $\mathcal M$ is a type I von Neumann algebra,
in \cite{Albeverio1}, the authors prove that every derivation
on $LS(\mathcal M)$ is an inner derivation if and only if it is $\mathcal Z(\mathcal M)$ linear;
in \cite{Albeverio2}, the authors give the decomposition form of derivations on $S(\mathcal M)$
and $LS(\mathcal M)$; they also prove that if $\mathcal M$ is a type $\mathrm{I}_\infty$ von Neumann algebra,
then every derivation on $S(\mathcal M)$ or $LS(\mathcal M)$ is an inner derivation.
If $\mathcal M$ is a properly infinite von Neumann algebra,
in \cite{Ber3}, Ber, Chilin and Sukochev prove that
every derivation on $LS(\mathcal M)$ is continuous with respect
to the local measure topology $t(\mathcal M)$; and in \cite{Ber2}, the authors
show that every derivation on $LS(\mathcal M)$ is an inner derivation.

Similar to the derivations and Jordan derivations, in \cite{M. Bresar J. Vukman},
M. Bre\v{s}ar and J. Vukman introduce the concepts of left derivations and Jordan left derivations.

Let $\mathcal{R}$ be a ring and $\mathcal{M}$ be a left $\mathcal{R}$-module.
A linear mapping $\delta$ from $\mathcal{R}$ into $\mathcal{M}$ is
called a \emph{left derivation} if
$\delta(ab)=a\delta(b)+b\delta(a)$
for each $a,b$ in $\mathcal{R}$; and $\delta$ is called a \emph{Jordan left derivation}
if $\delta(a^{2})=2a\delta(a)$ for every $a$ in $\mathcal{R}$.

In \cite{J. Vukman1}, Vukman shows that every Jordan left derivation from a
complex semisimple Banach algebra into itself is zero.
In \cite{An}, we prove that every Jordan left derivation from a $C^*$-algebra into its Banach left module is zero.

If $\mathcal M$ is a type $\mathrm{I}_n$ von Neumann algebra,
in \cite{Albeverio2}, Albeverio, Ayupov and Kudaybergenov show
that $S(\mathcal M)$ (resp. $LS(\mathcal M)$) is $*$-isomorphic to a
matrix algebra, by \cite[Theorem 3.1]{Alizadeh}, we know that
every Jordan derivation on $S(\mathcal M)$ (resp. $LS(\mathcal M)$) is a derivation.
If $\mathcal M$ is a type $\mathrm{III}$ von Neumann algebra, in \cite{M. Muratov},
Murator and Chilin show that $S(\mathcal M)=\mathcal M$, then every Jordan derivation on $S(\mathcal M)$
is an inner derivation.

This paper is organized as follows. In Section 2, we introduce some notations
and recall the definitions of measurable operator algebras and local measurable operator algebras.

In Section 3, we prove that if $\mathcal M$ is a properly
infinite von Neumann algebra, then every Jordan derivation
from $LS(\mathcal M)$ into itself is continuous with respect
to the local measure topology $t(\mathcal M)$.

In Section 4, we construct an extension of a Jordan derivation
from $\mathcal M$ into $LS(\mathcal M)$ up to a Jordan
derivation from $LS(\mathcal M)$ into itself. Moreover, we prove that
if $\mathcal M$ is a properly von Neumann algebra
and $\mathcal A$ is a subalgebra of $LS(\mathcal M)$
such that $\mathcal M\subset\mathcal A$, then every Jordan derivation
from $\mathcal A$ into $LS(\mathcal M)$ is continuous with respect to the local
measure topology $t(\mathcal M)$.

\section{Preliminaries}
Let $\mathcal H$ be a complex Hilbert space
and $B(\mathcal H)$ be the algebra of all bounded linear operators on $\mathcal H$.
Suppose that $\mathcal{M}$ is a von Neumann algebra on
$\mathcal{H}$ and
$\mathcal Z(\mathcal M)=\mathcal M\cap\mathcal M'$
is the center of $\mathcal M$, where
$$\mathcal M'=\{a\in B(\mathcal H):ab=ba~\mathrm{for~every}~b~\mathrm{in}~\mathcal M\}.$$
Denote by
$\mathcal P(\mathcal M)=\{p\in\mathcal M:p=p^*=p^2\}$
the lattice of all projections in $\mathcal M$ and by $\mathcal P_{fin}(\mathcal M)$
the set of all finite projections in $\mathcal M$. For each $p$ and $q$
in $\mathcal P(\mathcal M)$, if we define the
inclusion relation $p\subset q$ by $p\leq q$, then
$\mathcal P(\mathcal M)$ is a complete lattice. Suppose that
$\{p_l\}_{l\in\lambda}$ is a family of projections in $\mathcal M$, we denote by
$${\sup_{{l\in\lambda}}}p_l=\overline{\bigcup_{l\in\lambda} p_l\mathcal H}~~\mathrm{and}~~
{\inf_{{l\in\lambda}}}p_l=\bigcap_{l\in\lambda} p_l\mathcal H,$$
if $\{p_l\}_{l\in\lambda}$
is an orthogonal family of projections in $\mathcal M$, then we have that
$${\sup_{{l\in\lambda}}}p_l=\sum_{l\in\lambda}p_l.$$

Let $x$ be a closed densely defined linear operator on $\mathcal H$
with the domain $\mathcal D(x)$, where $\mathcal D(x)$ is a linear subspace of $\mathcal H$.
$x$ is said to be \emph{affiliated} with $\mathcal M$,
denote by $x\eta\mathcal M$, if
$u^*xu=x$ for every unitary element $u$ in $\mathcal{M}'$.

A closed densely defined linear operator $x\eta\mathcal M$
is said to be \emph{measurable} with respect to
$\mathcal M$, if there exists a sequence $\{p_n\}_{n=1}^{\infty}\subset\mathcal P(\mathcal M)$ such that $p_n\uparrow1$,
$p_n(\mathcal H)\subset\mathcal D(x)$ and $p_n^{\bot}=1-p_n\in \mathcal P_{fin}(\mathcal M)$ for every $n\in\mathbb{N}$,
where $\mathbb{N}$ is the set of all natural numbers. Denote by $S(\mathcal M)$ the set of all measurable operators
affiliated with the von Neumann algebra $\mathcal M$.

A closed densely defined linear operator $x\eta\mathcal M$
 is said to be \emph{locally measurable} with respect to
$\mathcal M$, if there exists a sequence $\{z_n\}_{n=1}^{\infty}\subset\mathcal P(\mathcal Z(\mathcal M))$
such that $z_n\uparrow1$ and $z_nx\in S(\mathcal M)$ for every $n\in\mathbb{N}$.
Denote by $LS(\mathcal M)$ the set of all locally measurable operators
affiliated with the von Neumann algebra $\mathcal M$.

In \cite{M. Muratov}, Murator and Chilin prove that $S(\mathcal M)$ and $LS(\mathcal M)$
are both unital $*$-algebras and $\mathcal M\subset S(\mathcal M)\subset LS(\mathcal M)$;
the authors also show that if $\mathcal M$ is a finite von Neumann algebra
or $\mathrm{dim}(\mathcal Z(\mathcal M))<\infty$, then
$S(\mathcal M)=LS(\mathcal M)$; if $\mathcal M$ is a type III von Neumann algebra and
$\mathrm{dim}(\mathcal Z(\mathcal M))=\infty$, then $S(\mathcal M)=\mathcal M$ and
$LS(\mathcal M)\neq\mathcal M$.

Let $E$ be a subset of $LS(\mathcal M)$. Denote by $E_h$ and $E_+$
the sets of all self-adjoint elements and all positive elements
in $E$, respectively. The partial order in $LS(\mathcal M)$
is defined by its cone $LS_+(\mathcal M)$ and is denoted by $\leq$.

Suppose that $x$ is a closed operator with a dense domain $\mathcal D(x)$ in $\mathcal H$.
Let $x=u|x|$ be the polar decomposition of $x$, where $|x|=(x^*x)^{\frac{1}{2}}$ and
$u$ is a partial isometry in $B(\mathcal H)$. Denote by $l(x)=uu^*$ the
\emph{left support} of $x$ and by $r(x)=u^*u$ the
\emph{right support} of $x$, clearly, $l(x)\sim u(x)$.
In \cite{M. Muratov},  Muratov and Chilin show that $x\in S(\mathcal M)$ (resp. $x\in LS(\mathcal M)$)
if and only if $|x|\in S(\mathcal M)$ (resp. $|x|\in LS(\mathcal M)$)
and $u\in\mathcal M$. If $x$ is a self-adjoint operator affiliated with $\mathcal M$,
then the spectral family of projections $\{E_{\lambda}(x)\}_{\lambda\in\mathbb{R}}$
for $x$ belongs to $\mathcal M$. In \cite{M. Muratov}, the authors also prove that a locally measurable operator
$x$ is measurable if and only if there exists $\lambda>0$ such that $E_\lambda^\bot(|x|)$
is a finite projection in $\mathcal M$.

In the following, we recall the definition of the local
measure topology.
Let $\mathcal M$ be a commutative von Neumann algebra,
in \cite{M. Takesaki}, Takesaki proves that there exists a $*$-isomorphism from
$\mathcal M$ onto the $*$-algebra
$L^\infty(\Omega,\Sigma,\mu)$, where $\mu$ is a measure
satisfying the direct sum property. The direct sum property
means that the Boolean algebra of all projections
in $L^\infty(\Omega,\Sigma,\mu)$ is total order, and for every
non-zero projection $p$ in $\mathcal M$, there exists a non-zero projection $q\leq p$
with $\mu(q)<\infty$. Consider $LS(\mathcal M)=S(\mathcal M)=L^0(\Omega,\Sigma,\mu)$
of all measurable almost everywhere finite complex-valued functions on $(\Omega,\Sigma,\mu)$.
Define the local measure topology $t(L^\infty(\Omega))$ on $L^0(\Omega,\Sigma,\mu)$, that is,
the Hausdorff vector topology, whose base of neighborhoods of zero is given by
\begin{align*}
W(B,\varepsilon,\delta)=&\{f\in L^0(\Omega,\Sigma,\mu): \mathrm{there~exists~a~set}~E\in\Sigma~\mathrm{such~that}\\
&E\subset B, \mu(B\backslash E)\leq\delta,f_{\chi_{E}}\in L^\infty(\Omega,\Sigma,\mu),\|f_{\chi_{E}}\|_{L^\infty(\Omega,\Sigma,\mu)}\leq\varepsilon\},
\end{align*}
where $\varepsilon,\delta>0,B\in\Sigma,\mu(B)<\infty$ and $\chi_{E}(\omega)=1$ when $\omega\in E$, $\chi_{E}(\omega)=0$ when $\omega\notin E$.
Suppose that $\{f_\alpha\}\subset L^0(\Omega,\Sigma,\mu)$ and $f\in L^0(\Omega,\Sigma,\mu)$, if
$f_\alpha\chi_B\rightarrow f\chi_B$ in the measure $\mu$ for every $B\in\Sigma$ with $\mu(B)<\infty$,
then we denote by
$f_\alpha\xrightarrow{t(L^\infty(\Omega))}f$. In \cite{F. Yeadon}, Yeadon show the topology
$t(L^\infty(\Omega))$ dose not change if the measure $\mu$ is replaced with an equivalent measure.

If $\mathcal M$ is an arbitrary von Neumann algebra
and $\mathcal Z(\mathcal M)$
is the center of $\mathcal M$, then
there exists a $*$-isomorphism $\varphi$ from
$\mathcal Z(\mathcal M)$ onto the $*$-algebra
$L^\infty(\Omega,\Sigma,\mu)$, where $\mu$ is a measure
satisfying the direct sum property.
Denote by $L^+(\Omega,\Sigma,\mu)$
the set of all measurable real-valued positive functions on $(\Omega,\Sigma,\mu)$.
In \cite{Segal}, Segal shows that there exists a mapping
$\Delta$ from $\mathcal P(\mathcal M)$ into $L^+(\Omega,\Sigma,\mu)$ satisfying the following conditions:\\
($\mathbb{D}_{1}$) $\Delta(p)\in L_+^0(\Omega,\Sigma,\mu)$ if and only if $p\in \mathcal P_{fin}(\mathcal M)$;\\
($\mathbb{D}_{2}$) $\Delta(p\vee q)=\Delta(p)+\Delta(q)$ if $pq=0$;\\
($\mathbb{D}_{3}$) $\Delta(u^*u)=\Delta(uu^*)$ for every partial isometry $u\in\mathcal M$;\\
($\mathbb{D}_{4}$) $\Delta(zp)=\varphi(z)\Delta(p)$ for every $z\in\mathcal P(\mathcal Z(\mathcal M))$
and every $p\in\mathcal P(\mathcal M)$;\\
($\mathbb{D}_{5}$) if $p_\alpha,p\in\mathcal P(\mathcal M),\alpha\in\Gamma$ and $p_\alpha\uparrow p$, then $\Delta(p)=\mathrm{sup}_{\alpha\in\Gamma}\Delta(p_\alpha)$.\\
In addition, $\Delta$ is called a \emph{dimension function} on $\mathcal P(\mathcal M)$
and $\Delta$ also satisfies the following two conditions:\\
($\mathbb{D}_{6}$) if $\{p_n\}_{n=1}^\infty\subset\mathcal P(\mathcal M)$,
then $\Delta(\mathrm{sup}_{n\geq1}p_n)\leq\sum_{n=1}^{\infty}\Delta(p_n)$; moreover, if $p_np_m=0$
when $n\neq m$, then $\Delta(\mathrm{sup}_{n\geq1}p_n)=\sum_{n=1}^{\infty}\Delta(p_n)$;\\
($\mathbb{D}_{7}$) if $\{p_n\}_{n=1}^\infty\subset\mathcal P(\mathcal M)$ and $p_n\downarrow0$, then $\Delta(p_n)\rightarrow0$ almost everywhere.

For arbitrary scalars $\varepsilon,\gamma>0$ and a set $B\in\Sigma$, $\mu(B)<\infty$, we let
\begin{align*}
V(B,\varepsilon,\gamma)=&\{x\in LS(\mathcal M):~\mathrm{there~exist}~p\in\mathcal P(\mathcal M)~\mathrm{and}~z\in\mathcal P(Z(\mathcal M)~\mathrm{such~that}\\&xp\in\mathcal M, \|xp\|_{\mathcal M}\leq\varepsilon,\varphi(z^\bot)\in W(B,\varepsilon,\gamma)~\mathrm{and}~\Delta(zp^\bot)\leq\varepsilon\varphi(z)\},
\end{align*}
where $\|\cdot\|_{\mathcal M}$ is the $C^*$-norm on $\mathcal M$.
In \cite{F. Yeadon}, Yeadon shows that the system of sets
$$\{x+V(B,\varepsilon,\gamma):x\in LS(\mathcal M),\varepsilon,\gamma>0,B\in\Sigma~\mathrm{and}~\mu(B)<\infty\}$$
defines a Hausdorff vector topology $t(\mathcal M)$ on $LS(\mathcal M)$ and the sets $$\{x+V(B,\varepsilon,\gamma),\varepsilon,\gamma>0,B\in\Sigma~\mathrm{and}~\mu(B)<\infty\}$$
form a neighborhood base of a local measurable operator $x$ in $LS(\mathcal M)$.
In \cite{F. Yeadon}, Yeadon also proves that $(LS(\mathcal M),t(\mathcal M))$ is
a complete topological $*$-algebra, and the topology $t(\mathcal M)$
does not depend on the choices of dimension function $\Delta$ and $*$-isomorphism $\varphi$.
The topology $t(\mathcal M)$ on $LS(\mathcal M)$ is called the \emph{local measure topology}.
Moreover, if $\mathcal M=B(\mathcal H)$, then $LS(\mathcal M)=\mathcal M$ and the
local measure topology topology $t(\mathcal M)$
coincides with the uniform topology $\|\cdot\|_{B(\mathcal H)}$.

\section{Continuity of Jordan derivations on algebras of locally measurable operators }

In this section, we give a characterization of Jordan derivations on $LS(\mathcal M)$ of all locally
measurable operators affiliated with a von Neumann algebra $\mathcal M$. We show that if $\mathcal M$ is a properly
infinite von Neumann algebra, then every Jordan derivation from $LS(\mathcal M)$ into itself is
continuous with respect to the local measure topology $t(\mathcal M)$.

To prove the main theorem, we need the following lemmas.

\begin{lemma}
Suppose that $\mathcal{M}$ is a von Neumann algebra
and $\mathcal A$ is a subalgebra of $LS(\mathcal M)$
such that $\mathcal P(\mathcal Z(\mathcal M))\subset\mathcal A$.
If $\delta$ is a Jordan
derivation from $\mathcal{A}$ into $LS(\mathcal{M})$, then for every $a$ in $\mathcal{A}$
and every central projection $z$ in $\mathcal M$,
we have the following two statements:\\
$(1)$ $\delta(ab+ba)=\delta(a)b+a\delta(b)+\delta(b)a+b\delta(a)$;\\
$(2)$ $\delta(z)=0;$\\
$(3)$ $\delta(za)=z\delta(a).$
\end{lemma}

Suppose that $\mathcal M$ is a von Neumann algebra and
$\eta$ is a linear mapping from $LS(\mathcal M)$ into itself
such that $\eta (zx)=z\eta(x)$ for every $z$ in $\mathcal P(\mathcal Z(\mathcal M))$.
Thus we can define a linear mapping $\eta_z$ from $LS(z\mathcal M)$
into $LS(z\mathcal M)$ by
$$\eta_z(y)=\eta(y)$$
for every $y$ in $LS(z\mathcal M)$. By \cite[Proposition 2.6]{Ber3}, we can obtain the following
result.

\begin{theorem}\cite{Ber3}
Suppose that $\mathcal M$ is a von Neumann algebra and $\{z_i\}_{i\in I}$ is a orthogonal family of
central projections such that $\mathrm{sup}_{i\in I}z_i=1$.
If $\eta$ is a linear mapping from $LS(\mathcal M)$ into itself such that
$\eta (z_ix)=z_i\eta(x)$ for every $x\in LS(\mathcal M)$ and every $i\in I$, then the following two statements are equivalent:\\
$(1)$ The mapping $\eta:(LS(\mathcal M),t(\mathcal M))\rightarrow(LS(\mathcal M),t(\mathcal M))$ is continuous;\\
$(2)$ The mapping $\eta_{z_i}:(LS(z_i\mathcal M),t(z_i\mathcal M))\rightarrow(LS(z_i\mathcal M),t(z_i\mathcal M))$ is continuous for every $i\in\Gamma$.
\end{theorem}

Suppose that $\mathcal M$ is a von Neumann algebra. The projection lattice
$\mathcal P(\mathcal M)$ is said to have a \emph{countable type}, if every family
of non-zero pairwise orthogonal projections in $\mathcal P(\mathcal M)$ is
countable.

If $\mathcal M$
is a commutative von Neumann algebra and $\mathcal P(\mathcal M)$
has a countable type, then $\mathcal M$ is $*$-isomorphic to a $*$-algebra
$L^\infty(\Omega,\Sigma,\mu)$ with $\mu(\Omega)<\infty$. Hence the topology
$t(L^\infty(\Omega))$ is metrizable  and has a base of neighborhoods of $0$
consisting of the sets $W(\Omega,1/n,1/n)$, $n\in\mathbb{N}$.

If $\mathcal M$
is a commutative von Neumann algebra and $\mathcal P(\mathcal M)$
does not have a countable type. Denote by $\varphi$ a $*$-isomorphism from
$\mathcal M$ onto $L^{\infty}(\Omega,\Sigma,\mu)$, where $\mu$ is a measure such that
the direct sum property. Hence there exists a family of non-zero pairwise orthogonal
projections $\{p_i\}_{i\in\Gamma}\subset\mathcal P(\mathcal M)$ such that
$\mathrm{sup}_{i\in\Gamma}p_i=1$ and $\mu(\varphi(p_i))<\infty$, in particular,
$\mathcal P(p_i(\mathcal Z(\mathcal M)))$ has a countable type.

Let $\mathcal A$ be a $*$-subalgebra of $LS(\mathcal M)$ and $\delta$ be a Jordan derivation from
$\mathcal A$ into $LS(\mathcal M)$. We can define a linear mapping $\delta^*$ from $\mathcal A$ into $LS(\mathcal M)$
by
$$\delta^*(a)=(\delta(a^*))^*$$
for every $a$ in $\mathcal A$. It is easy to show that $\delta^*$
is also a Jordan derivation.

A Jordan derivation $\delta$ from
$\mathcal A$ into $LS(\mathcal M)$ is said to be \emph{self-adjoint}
if $\delta=\delta^*$. Moreover, every Jordan derivation $\delta$ can be represented in the form
$$\delta=\mathrm{Re}(\delta)+i\mathrm{Im}(\delta),$$
where
$\mathrm{Re}(\delta)=(\delta+\delta^*)/2$ is the real part of $\delta$
and $\mathrm{Im}(\delta)=(\delta-\delta^*)/2i$ is the imaginary part of $\delta$.
Obviously, $\mathrm{Re}(\delta)$ and $\mathrm{Im}(\delta)$ are also two Jordan derivations.

Since $(LS(\mathcal M),t(\mathcal M))$ is a topological $*$-algebra, we have the following result.

\begin{lemma}
Suppose that $\mathcal M$ is a von Neumann algebra
and $\mathcal A$ is a subalgebra of $LS(\mathcal M)$. If $\delta$ is a
Jordan derivation $\delta$ from $\mathcal A$ into $LS(\mathcal M)$, then $\delta$ is continuous with respect
to the topology $t(\mathcal M)$ if and only if the self-adjoint Jordan derivations $\mathrm{Re}(\delta)$
and $\mathrm{Im}(\delta)$ are continuous with respect to the topology $t(\mathcal M)$.
\end{lemma}

The following theorem is the main result in this section.

\begin{theorem}
Suppose that $\mathcal{M}$ is a properly infinite von Neumann algebra and
$\delta$ is a Jordan derivation from $LS(\mathcal{M})$ into itself. Then $\delta$ is
automatically continuous with respect the local measure topology $t(\mathcal{M})$.
\end{theorem}

\begin{proof}
By Lemma 3.3, we can assume that $\delta$ is a self-adjoint Jordan derivation,
that is $\delta=\delta^*$.

Since $\mathcal Z(\mathcal{M})$ is a commutative von Neumann algebra,
there exists an orthogonal family of central projections
$\{z_i\}_{i\in\Gamma}\subset\mathcal{P}(\mathcal Z(\mathcal{M}))$ such that
$\mathrm{sup}_{i\in\Gamma}z_i=1$ and the Boolean algebra $\mathcal{P}(z_i\mathcal Z(\mathcal{M}))$ has a countable
type for every $i\in\Gamma$.
By Lemma 3.2 (2), we have that
$\delta(z_ix)=z_i\delta(x)$
for every $x$ in $LS(\mathcal{M})$ and every $i$ in $\Gamma$, by Theorem 3.2,
it is sufficient to prove that $\delta_{z_i}$ on $LS(z_i\mathcal M)$ is continuous with the local measure topology $t(z_i\mathcal{M})$
for every $i\in I$.
Thus without loss of generality, we can assume that the
Boolean algebra $\mathcal{P}(\mathcal Z(\mathcal{M}))$ has a countable type. It follows that
the topology $t(\mathcal M)$ is metrizable, and the sets $V(\Omega,1/n,1/n),n\in\mathbb{ N}$
form a countable base of neighborhoods of $0$, in particular, $(LS(\mathcal M),t(\mathcal M))$ is an
$F$-space.

By contradiction, we suppose that $\delta$ is discontinuous with respect the local measure topology $t(\mathcal{M})$. It means that
there exist a sequence $\{a_n\}\subset LS(\mathcal M)$ and a non-zero element
$a\in LS(\mathcal M)$ such that
$a_n\xrightarrow[n\rightarrow\infty]{t(\mathcal M)}0$ and $\delta(a_n)\xrightarrow[n\rightarrow\infty]{t(\mathcal M)}a$.
Since $(LS(\mathcal M),t(\mathcal M))$ is a topological $*$-algebra and $\delta$ is a self-adjoint
Jordan derivation, we can assume that $a_n=a_n^*$ and $a=a^*$ for every $n$ in $\mathbb{N}$.
Hence $a=a_+-a_-$, where $a_+$ and $a_-$ are positive and
negative parts of $a$ in $LS_+(\mathcal M)$, respectively,
we can assume that $a_+\neq0$,
otherwise, we replace $a_n$ with $-a_n$.

Now we select two scalars $0<\lambda_1<\lambda_2$ such that the projection
$p_a=E_{\lambda_2}(a)-E_{\lambda_1}(a)\neq0$.
It implies that
$0<\lambda_1p_a\leq p_aap_a\leq\lambda_2p_a~\mathrm{and}~\|p_aap_a\|_{\mathcal{M}}\leq\lambda_2$,
it means that $p_aap_a$ is a positive element in $\mathcal M$.

Next we denote by $x_n=p_aa_np_a$ and $x=p_aap_a$, by Lemma 3.1 (3) we have that
$$\delta(x_n)=\delta(p_aa_np_a)=\delta(p_a)a_np_a+p_a\delta(a_n)p_a+p_aa_n\delta(p_a),$$
it follows that $x_n\xrightarrow[n\rightarrow\infty]{t(\mathcal M)}0$ and $\delta(x_n)\xrightarrow[n\rightarrow\infty]{t(\mathcal M)}x$.

It is obvious that $x$ is a positive element in $\mathcal M$, we can select two scalars $0<\mu_1<\mu_2$ such that the projection
$p_x=E_{\mu_2}(x)-E_{\mu_1}(x)\neq0$.
It implies that
$$0<\mu_1p_x\leq p_xxp_x\leq\mu_2p_x~\mathrm{and}~\|p_xxp_x\|_{\mathcal{M}}\leq\mu_2.$$
Without loss of generality, we can assume that $\mu_1=1$, then
$p_xxp_x\geq p_x$,
otherwise, we replace $x_n$ with $x_n/\mu_1$,
Since $\mathcal M$ is a properly infinite von Neumann algebra, there exists a sequence
of non-zero pairwise orthogonal projections
$\{\hat{p_m}\}_{m=1}^{\infty}$ in $\mathcal M$ such that $\sum_{m=1}^{\infty}\hat{p_m}=1$,
$\hat{p_m}\sim1$ and $p_x\preceq\hat{p_m}$ for every $m\in\mathbb{N}$,
it means that there exists a sequence of non-zero pairwise orthogonal projections
$\{p_m\}_{m=1}^{\infty}$ in $\mathcal M$ such that $p_x\sim p_m\leq\hat{p_m}$
for every $m\in\mathbb{N}$.
It follows that there is a partial isometry $v_m$ in $\mathcal M$ such that $v_m^*v_m=p_x$
and $v_mv_m^*=p_m$. Denote by
$$p_0=\sum_{m=1}^{\infty}p_m,$$
it is easy to see that $p_0$ is an infinite projection in $\mathcal M$.

Since $x$ is a positive element in $\mathcal M$, we know that $p_mv_mxv_m^*p_m$ and $p_mv_m^*xv_mp_m$
are also two positive elements in $\mathcal M$, and by $p_xxp_x\geq p_x$, we have that
$$p_m(v_mxv_m^*+v_m^*xv_m)p_m\geq p_mv_mxv_m^*p_m=p_mv_mp_xxp_xv_m^*p_m\geq p_mv_mp_xv_m^*p_m=p_m$$
and
\begin{align*}
\|p_m(v_mxv_m^*+v_m^*xv_m)p_m\|_{\mathcal M}&\leq\|p_mv_mxv_m^*p_m\|_{\mathcal M}+\|p_mv_m^*xv_mp_m\|_{\mathcal M}\notag\\
&\leq2\|x\|_{\mathcal M}\leq2\lambda_2.
\end{align*}
It means that the series $\sum_{m=1}^\infty p_m(v_mxv_m^*+v_m^*xv_m)p_m$ converges with respect to
the topology $\tau_{so}$ to some operator $y$ in $\mathcal M$ satisfying
\begin{align}
\|y\|_{\mathcal M}=\mathrm{sup}_{m\geq1}\|p_m(v_mxv_m^*+v_m^*xv_m)p_m\|_{\mathcal M}\leq2\lambda_2~~\mathrm{and}~~y\geq p_0.         \label{301}
\end{align}
In the following, we assume that the central carrier $\mathcal C_{p_0}$ of the projection $p_0$ is equal to $1$, otherwise,
we replace the algebra $\mathcal M$ with the algebra $\mathcal C_{p_0}\mathcal M\mathcal C_{p_0}$.

Let $\varphi$ be a $*$-isomorphism from $\mathcal Z(\mathcal M)$ onto $L^\infty(\Omega,\Sigma,\mu)$.
By the assumption, the Boolean algebra $\mathcal P(\mathcal Z(\mathcal M))$ has a countable type, and
we assume that $\mu(\Omega)=\int_{\Omega}1_{L^\infty(\Omega)}d\mu=1$, where $1_{L^\infty(\Omega)}$
is the identity of the $*$-algebra $L^\infty(\Omega,\Sigma,\mu)$. In this case, the countable base of
neighborhoods of $0$ in the topology $t(\mathcal M)$ is formed by the sets $V(\Omega,1/n,1/n)$, $n\in\mathbb{N}$.

By $x_n\xrightarrow[n\rightarrow\infty]{t(\mathcal M)}0$ and $\delta(x_n)\xrightarrow[n\rightarrow\infty]{t(\mathcal M)}x$,
we can obtain the following four inequalities:
\begin{align}
p_m(v_mx_nv_m^*+v_m^*x_nv_m)p_m\xrightarrow[n\rightarrow\infty]{t(\mathcal M)}0;                              \label{302}
\end{align}
\begin{align}
\delta(p_m)(v_mx_nv_m^*+v_m^*x_nv_m)p_m\xrightarrow[n\rightarrow\infty]{t(\mathcal M)}0;                              \label{303}
\end{align}
\begin{align}
p_m(\delta(v_m)x_nv_m^*+\delta(v_m^*)x_nv_m)p_m\xrightarrow[n\rightarrow\infty]{t(\mathcal M)}0;                              \label{304}
\end{align}
and
\begin{align}
p_m(v_m\delta(x_n)v_m^*+v_m^*\delta(x_n)v_m)p_m\xrightarrow[n\rightarrow\infty]{t(\mathcal M)}p_m(v_mxv_m^*+v_m^*xv_m)p_m.               \label{305}
\end{align}

In the following fix $k$ in $\mathbb{N}$. By \eqref{302},
we can select an index $n_1(m,k)$, a projection $q^{(1)}_{m,n}\in\mathcal P(\mathcal M)$
and a central projection $z^{(1)}_{m,n}\in\mathcal P(\mathcal Z(\mathcal M))$, such that the following
three  inequalities:
\begin{align}
\|p_m(v_mx_nv_m^*+v_m^*x_nv_m)p_mq^{(1)}_{m,n}\|_{\mathcal M}\leq2^{-m}(k+1)^{-1};    \label{306}
\end{align}
\begin{align}
\int_\Omega\varphi(1-z^{(1)}_{m,n})d\mu\leq4^{-1}2^{-m-k-1};                          \label{307}
\end{align}
and
\begin{align}
\Delta(z^{(1)}_{m,n}(1-q^{(1)}_{m,n}))\leq4^{-1}2^{-m-k-1}\varphi(z^{(1)}_{m,n})     \label{308}
\end{align}
for every $n\geq n_1(m,k)$.

By \eqref{303}, we can select an index $n_2(m,k)$, a projection $q^{(2)}_{m,n}\in\mathcal P(\mathcal M)$
and a central projection $z^{(2)}_{m,n}\in\mathcal P(\mathcal Z(\mathcal M))$, such that the following
three  inequalities:
\begin{align}
\|\delta(p_m)(v_mx_nv_m^*+v_m^*x_nv_m)p_mq^{(2)}_{m,n}\|_{\mathcal M}\leq5^{-1}2^{-m}(k+1)^{-1};    \label{309}
\end{align}
\begin{align}
\int_\Omega\varphi(1-z^{(2)}_{m,n})d\mu\leq4^{-1}2^{-m-k-1};                          \label{310}
\end{align}
and
\begin{align}
\Delta(z^{(2)}_{m,n}(1-q^{(2)}_{m,n}))\leq4^{-1}2^{-m-k-1}\varphi(z^{(2)}_{m,n})     \label{311}
\end{align}
for every $n\geq n_2(m,k)$.

By \eqref{304} we can select an index $n_3(m,k)$, a projection $q^{(3)}_{m,n}\in\mathcal P(\mathcal M)$
and a central projection $z^{(3)}_{m,n}\in\mathcal P(\mathcal Z(\mathcal M))$, such that the following
three  inequalities:
\begin{align}
\|p_m(\delta(v_m)x_nv_m^*+\delta(v_m^*)x_nv_m)p_mq^{(3)}_{m,n}\|_{\mathcal M}\leq5^{-1}2^{-m}(k+1)^{-1};    \label{312}
\end{align}
\begin{align}
\int_\Omega\varphi(1-z^{(3)}_{m,n})d\mu\leq4^{-1}2^{-m-k-1};                          \label{313}
\end{align}
and
\begin{align}
\Delta(z^{(3)}_{m,n}(1-q^{(3)}_{m,n}))\leq4^{-1}2^{-m-k-1}\varphi(z^{(3)}_{m,n})     \label{314}
\end{align}
for every $n\geq n_3(m,k)$.

By \eqref{305}, we can select an index $n_4(m,k)$, a projection $q^{(4)}_{m,n}\in\mathcal P(\mathcal M)$
and a central projection $z^{(4)}_{m,n}\in\mathcal P(\mathcal Z(\mathcal M))$, such that the following
three  inequalities:
\begin{align}
\|p_m((v_m\delta(x_n)v_m^*+v_m^*\delta(x_n)v_m)-&(v_mxv_m^*+v_m^*xv_m))p_mq^{(4)}_{m,n}\|_{\mathcal M}       \notag\\
\leq&5^{-1}2^{-m}(k+1)^{-1};                                                                                   \label{315}
\end{align}
\begin{align}
\int_\Omega\varphi(1-z^{(4)}_{m,n})d\mu\leq4^{-1}2^{-m-k-1};                          \label{316}
\end{align}
and
\begin{align}
\Delta(z^{(4)}_{m,n}(1-q^{(4)}_{m,n}))\leq4^{-1}2^{-m-k-1}\varphi(z^{(4)}_{m,n})     \label{317}
\end{align}
for every $n\geq n_4(m,k)$.

Let $n(m,k)=\mathrm{max}_{i=1,2,3,4}n_i(m,k)$, $q_m=\mathrm{inf}_{i=1,2,3,4}q^{(i)}_{m,n}$
and $z_m=\mathrm{inf}_{i=1,2,3,4}z^{(i)}_{m,n}$. By \eqref{306}, \eqref{309}, \eqref{312} and \eqref{315},
we have that the following four inequalities:
\begin{align}
\|p_m(v_mx_{n(m,k)}v_m^*+v_m^*x_{n(m,k)}v_m)p_mq_{m}\|_{\mathcal M}\leq2^{-m}(k+1)^{-1};           \label{318}
\end{align}
\begin{align}
&\|\delta(p_m)(v_mx_{n(m,k)}v_m^*+v_m^*x_{n(m,k)}v_m)p_mq_{m}\|_{\mathcal M}  \notag\\
=&\|q_{m}p_m(v_mx_{n(m,k)}v_m^*+v_m^*x_{n(m,k)}v_m)\delta(p_m)\|_{\mathcal M} \notag\\
\leq&5^{-1}2^{-m}(k+1)^{-1};                                                                       \label{319}
\end{align}
\begin{align}
&\|p_m(\delta(v_m)x_{n(m,k)}v_m^*+\delta(v_m^*)x_{n(m,k)}v_m)p_mq_{m}\|_{\mathcal M}         \notag\\
=&\|q_{m}p_m(v_mx_{n(m,k)}\delta(v_m^*)+v_m^*x_{n(m,k)}\delta(v_m))p_m\|_{\mathcal M}        \notag\\
\leq&5^{-1}2^{-m}(k+1)^{-1};                                                                            \label{320}
\end{align}
and
\begin{align}
\|p_m((v_m\delta(x_{n(m,k)})v_m^*+v_m^*\delta(x_{n(m,k)})v_m)-&(v_mxv_m^*+v_m^*xv_m))p_mq_{m}\|_{\mathcal M}       \notag\\
\leq&5^{-1}2^{-m}(k+1)^{-1}.                                                                                \label{321}
\end{align}
By \eqref{307}, \eqref{310}, \eqref{313} and \eqref{316}, we have that
\begin{align}
1-\int_\Omega\varphi(z_m)d\mu=\int_\Omega\varphi(1-z_m)d\mu&\leq\sum_{i=1}^4\int_\Omega\varphi(1-z^{(i)}_{m,n(m,k)})d\mu\notag\\
                                                           &\leq2^{-m-k-1}.                    \label{322}
\end{align}
By \eqref{308}, \eqref{311}, \eqref{314}, \eqref{317} and Condition $\mathbb{D}_6$ we have that
\begin{align}
\Delta(z_m(1-q_m))\leq\sum_{i=1}^4\Delta(z_m(1-q^{(i)}_{m,n(m,k)}))\leq2^{-m-k-1}\varphi(z_m).                   \label{323}
\end{align}

Let $m_1$ and $m_2$ be in $\mathbb{N}$ with $m_1<m_2$, denote by
$$q_{m_1,m_2}={\inf_{m_1<m\leq m_2}}q_m~~\mathrm{and}~~z_{m_1,m_2}={\inf_{m_1<m\leq m_2}}z_m.$$
Thus
$$1-q_{m_1,m_2}={\sup_{m_1<m\leq m_2}}1-q_m~~\mathrm{and}~~1-z_{m_1,m_2}={\sup_{m_1<m\leq m_2}}1-z_m.$$
It follows that
$$\varphi(1-q_{m_1,m_2})={\sup_{m_1<m\leq m_2}}\varphi(1-q_m)$$
and
$$\varphi(1-z_{m_1,m_2})={\sup_{m_1<m\leq m_2}}\varphi(1-z_m).$$
Hence by \eqref{322}, we can obtain that
\begin{align}
1-\int_\Omega\varphi(z_{m_1,m_2})d\mu=\int_\Omega\varphi(1-z_{m_1,m_2})d\mu&\leq\sum_{m=m_1+1}^{m_2}\int_\Omega\varphi(1-z_m)d\mu\notag\\
                                                                           &\leq2^{-m_1-k-1}.                                          \label{324}
\end{align}
By \eqref{323} and Condition $\mathbb{D}_{6}$, we can obtain that
\begin{align}
\Delta(z_{m_1,m_2}(1-q_{m_1,m_2}))&\leq\sum_{m=m_1+1}^{m_2}\Delta(z_{m_1,m_2}(1-q_m))  \notag\\
&\leq2^{-m_1-k-1}\varphi(z_{m_1,m_2}).                                                                    \label{325}
\end{align}
Moreover, by \eqref{318}, we have that
\begin{align}
&\|\sum^{m_2}_{m=m_1+1}p_m(v_mx_{n(m,k)}v_m^*+v_m^*x_{n(m,k)}v_m)p_mq_{m_1,m_2}\|_{\mathcal M}\notag\\
\leq&\sum^{m_2}_{m=m_1+1}\|p_m(v_mx_{n(m,k)}v_m^*+v_m^*x_{n(m,k)}v_m)p_mq_{m}\|_{\mathcal M}\notag\\
\leq&2^{-m_1}(k+1)^{-1}.                                                                                     \label{326}
\end{align}

By \eqref{324}, \eqref{325} and \eqref{326}, it follows that
$$S_{l,k}=\sum^{l}_{m=1}p_m(v_mx_{n(m,k)}v_m^*+v_m^*x_{n(m,k)}v_m)p_m,~~l\geq1$$
is a Cauchy sequence in the $F$-space $(LS(\mathcal M),t(\mathcal M))$
for every $k\in\mathbb{N}$.
Hence, there exists an element $y_k\in LS(\mathcal M)$ such that
$S_{l,k}\xrightarrow[l\rightarrow\infty]{t(\mathcal M)}y_k$.
Thus the series
\begin{align}
y_k=\sum^{\infty}_{m=1}p_m(v_mx_{n(m,k)}v_m^*+v_m^*x_{n(m,k)}v_m)p_m                      \label{327}
\end{align}
converges in $LS(\mathcal M)$ with respect to the topology $t(\mathcal M)$.
Since the involution is continuous in topology $t(\mathcal M)$ and $S_{l,k}=S_{l,k}^*$,
it implies that $y_k=y_k^*$.

Denote by
\begin{align}
r_m=p_m\wedge q_m,~~m\in\mathbb{N}.                      \label{328}
\end{align}
Since
$z_m(p_m-p_m\wedge q_m)\sim z_m(p_m\vee q_m-q_m)$, by Condition $\mathbb{D}_3$ and \eqref{323}
we have that
\begin{align}
\Delta(z_m(p_m-r_m))=&\Delta(z_m(p_m-p_m\wedge q_m))=\Delta(z_m(p_m\vee q_m-q_m)) \notag\\
\leq&\Delta(z_m(1-q_m))\leq2^{-m-k-1}\varphi(z_m).                                        \label{329}
\end{align}
Denote by
\begin{align}
q_0^{(k)}={\sup_{m\geq1}}r_m~~\mathrm{and}~~z_0^{(k)}={\inf_{m\geq1}}z_m,                       \label{330}
\end{align}
by \eqref{301}, \eqref{328} and \eqref{330}, we have that
\begin{align}
y\geq p_0\geq q_0^{(k)}.                                                               \label{331}
\end{align}
By \eqref{324}, we have that
\begin{align}
1-\int_{\Omega}\varphi(z_0^{(k)})d\mu=\int_{\Omega}\varphi(1-z_0^{(k)})d\mu\leq2^{-k-1}.                                     \label{332}
\end{align}

Since $p_mp_j=0$ when $m\neq j$ and $r_m\leq p_m$, by \eqref{330} we can obtain that
$p_0-q_0^{(k)}=\mathrm{sup}_{m\geq1}(p_m-r_m)$. By Condition $\mathbb{D}_6$ and \eqref{329} we have that
\begin{align}
\Delta(z_0^{k}(p_0-q_0^{(k)}))=\sum_{m=1}^\infty\Delta(z_0^{k}(p_m-r_m))\leq 2^{-k-1}\varphi(z_0^{k}).                                     \label{333}
\end{align}

By \eqref{328}, we have that $p_mq_0^{(k)}=r_mq_0^{(k)}=r_m=r_mq_m$ for every $m\in\mathbb{N}$. It follows that
$$p_m(v_mx_{n(m,k)}v_m^*+v_m^*x_{n(m,k)}v_m)p_mq_0^{(k)}=p_m(v_mx_{n(m,k)}v_m^*+v_m^*x_{n(m,k)}v_m)p_mr_m$$
and by \eqref{318}, we can obtain that
\begin{align}
\|y_kq_0^{(k)}\|_{\mathcal M}=&\|\sum^{\infty}_{m=1}p_m(v_mx_{n(m,k)}v_m^*+v_m^*x_{n(m,k)}v_m)p_mq_0^{(k)}\|_{\mathcal M}\notag\\
=&\|\sum^{\infty}_{m=1}p_m(v_mx_{n(m,k)}v_m^*+v_m^*x_{n(m,k)}v_m)p_mq_0^{(k)}\|_{\mathcal M}\notag\\
\leq&\sum^{\infty}_{m=1}\|p_m(v_mx_{n(m,k)}v_m^*+v_m^*x_{n(m,k)}v_m)p_mr_m\|_{\mathcal M}\notag\\
\leq&\sum^{\infty}_{m=1}\|p_m(v_mx_{n(m,k)}v_m^*+v_m^*x_{n(m,k)}v_m)p_mq_m\|_{\mathcal M}\notag\\
\leq&(k+1)^{-1}                                                                                               \label{334}
\end{align}

By Lemma 3.1 (2) and (3), we have that
\begin{align*}
&q_0^{(k)}\delta(p_m(v_mx_{n(m,k)}v_m^*+v_m^*x_{n(m,k)}v_m)p_m)q_0^{(k)}\\
=&q_0^{(k)}\delta(p_m)(v_mx_{n(m,k)}v_m^*+v_m^*x_{n(m,k)}v_m)p_mq_0^{(k)}+q_0^{(k)}p_m(v_mx_{n(m,k)}v_m^*+v_m^*x_{n(m,k)}v_m)\delta(p_m)q_0^{(k)}\\
&+q_0^{(k)}p_m\delta(v_mx_{n(m,k)}v_m^*+v_m^*x_{n(m,k)}v_m)p_mq_0^{(k)}\\
=&q_0^{(k)}\delta(p_m)(v_mx_{n(m,k)}v_m^*+v_m^*x_{n(m,k)}v_m)p_mq_0^{(k)}+q_0^{(k)}p_m(v_mx_{n(m,k)}v_m^*+v_m^*x_{n(m,k)}v_m)\delta(p_m)q_0^{(k)}\\
&+q_0^{(k)}p_m(\delta(v_m)x_{n(m,k)}v_m^*+\delta(v_m^*)x_{n(m,k)}v_m+v_mx_{n(m,k)}\delta(v_m^*)+v_m^*x_{n(m,k)}\delta(v_m))p_mq_0^{(k)}\\
&+q_0^{(k)}p_m((v_m\delta(x_{n(m,k)})v_m^*+v_m^*\delta(x_{n(m,k)})v_m)-(v_mxv_m^*+v_m^*xv_m))p_mq_0^{(k)}\\
&+q_0^{(k)}p_m(v_mxv_m^*+v_m^*xv_m)p_mq_0^{(k)}\\
=&q_0^{(k)}\delta(p_m)(v_mx_{n(m,k)}v_m^*+v_m^*x_{n(m,k)}v_m)p_mq_mr_m+r_mq_mp_m(v_mx_{n(m,k)}v_m^*+v_m^*x_{n(m,k)}v_m)\delta(p_m)q_0^{(k)}\\
&+r_mq_mp_m(\delta(v_m)x_{n(m,k)}v_m^*+\delta(v_m^*)x_{n(m,k)}v_m+v_mx_{n(m,k)}\delta(v_m^*)+v_m^*x_{n(m,k)}\delta(v_m))p_mq_mr_m\\
&+r_mq_mp_m((v_m\delta(x_{n(m,k)})v_m^*+v_m^*\delta(x_{n(m,k)})v_m)-(v_mxv_m^*+v_m^*xv_m))p_mq_mr_m\\
&+q_0^{(k)}p_m(v_mxv_m^*+v_m^*xv_m)p_mq_0^{(k)}.
\end{align*}

Consider the following formal series:
\begin{align}
\sum_{m=1}^\infty q_0^{(k)}\delta(p_m)(v_mx_{n(m,k)}v_m^*+v_m^*x_{n(m,k)}v_m)p_mq_mr_m;                                   \label{335}
\end{align}
\begin{align}
\sum_{m=1}^\infty r_mq_mp_m(v_mx_{n(m,k)}v_m^*+v_m^*x_{n(m,k)}v_m)\delta(p_m)q_0^{(k)};                                   \label{336}
\end{align}
\begin{align}
\sum_{m=1}^\infty r_mq_mp_m(\delta(v_m)x_{n(m,k)}v_m^*+\delta(v_m^*)x_{n(m,k)}v_m)p_mq_mr_m;                                   \label{337}
\end{align}
\begin{align}
\sum_{m=1}^\infty r_mq_mp_m(v_mx_{n(m,k)}\delta(v_m^*)+v_m^*x_{n(m,k)}\delta(v_m))p_mq_mr_m;                                   \label{338}
\end{align}
\begin{align}
\sum_{m=1}^\infty r_mq_mp_m&(v_m\delta(x_{n(m,k)})v_m^*+v_m^*\delta(x_{n(m,k)})v_m\notag\\
&-(v_mxv_m^*+v_m^*xv_m))p_mq_mr_m;                                                                                 \label{339}
\end{align}
and
\begin{align}
\sum_{m=1}^\infty q_0^{(k)}p_m(v_mxv_m^*+v_m^*xv_m)p_mq_0^{(k)}.                                   \label{340}
\end{align}

By \eqref{319}, we know that the first series \eqref{335} and the second series \eqref{336}
converge with respect to the norm $\|\cdot\|_{\mathcal M}$ to some elements $a$ and $b$ in $\mathcal M$,
respectively. Moreover, $\|a\|_{\mathcal M}\leq5^{-1}(k+1)^{-1}$ and $\|b\|_{\mathcal M}\leq5^{-1}(k+1)^{-1}$;
by \eqref{320}, we know that the third series \eqref{337} and the fourth series \eqref{338}
converge with respect to the norm $\|\cdot\|_{\mathcal M}$ to some elements $c$ and $d$ in $\mathcal M$;
respectively. Moreover, $\|c\|_{\mathcal M}\leq5^{-1}(k+1)^{-1}$ and $\|d\|_{\mathcal M}\leq5^{-1}(k+1)^{-1}$;
by \eqref{321}, we know that the fifth series \eqref{339}
converge with respect to the norm $\|\cdot\|_{\mathcal M}$ to some elements $e$ and
$\|e\|_{\mathcal M}\leq5^{-1}(k+1)^{-1}$.

Finally, since $y=\sum_{m=1}^\infty p_m(v_mxv_m^*+v_m^*xv_m)p_m$ (the convergence of the latter
series is taken in the topology), we see that the last series \eqref{340} converges with respect to the
topology $\tau_{so}$ to some element $q_0^{(k)}yq_0^{(k)}$. Hence the series
$$\sum_{m=1}^\infty q_0^{(k)}\delta(p_m(v_mx_{n(m,k)}v_m^*+v_m^*x_{n(m,k)}v_m)p_m)q_0^{(k)}$$
converges with respect to the topology $\tau_{so}$ to some element $a_k\in\mathcal M$, and in addition, we have that
\begin{align}
\|a_k-q_0^{(k)}yq_0^{(k)}\|_{\mathcal M}\leq(k+1)^{-1}.             \label{341}
\end{align}

Next we show that
$$a_k=q_0^{(k)}\delta(y_k)q_0^{(k)}$$
where $y_k=\sum_{m=1}^\infty p_m(v_mx_{n(m,k)}v_m^*+v_m^*x_{n(m,k)}v_m)p_m$.

Denote by $w_m=p_m(v_mx_{n(m,k)}v_m^*+v_m^*x_{n(m,k)}v_m)p_m$, for each $m_1$
and $m_2$ in $\mathbb{N}$, by \eqref{330}, we have that
\begin{align}
&r_{m_1}q_0^{(k)}\delta(y_k)q_0^{(k)}r_{m_2}+r_{m_2}q_0^{(k)}\delta(y_k)q_0^{(k)}r_{m_1}\notag\\
=&q_0^{(k)}(r_{m_1}\delta(y_k)r_{m_2}+r_{m_2}\delta(y_k)r_{m_1})q_0^{(k)}\notag\\
=&q_0^{(k)}(\delta(r_{m_1}y_kr_{m_2}+r_{m_2}y_kr_{m_1})-\delta(r_{m_1})y_kr_{m_2}-\delta(r_{m_2})y_kr_{m_1}\notag\\
&-r_{m_1}y_k\delta(r_{m_2})-r_{m_2}y_k\delta(r_{m_1}))q_0^{(k)}\notag\\
=&q_0^{(k)}(-\delta(r_{m_1})w_{m_2}r_{m_2}-\delta(r_{m_2})w_{m_1}r_{m_1}\notag\\
&-r_{m_1}w_{m_1}\delta(r_{m_2})-r_{m_2}w_{m_2}\delta(r_{m_1}))q_0^{(k)}                                                    \label{342}
\end{align}

Since the series $\sum_{m=1}^\infty q_0^{(k)}\delta(p_m(v_mx_{n(m,k)}v_m^*+v_m^*x_{n(m,k)}v_m)p_m)q_0^{(k)}$ converges
with respect to the topology $\tau_{so}$, it follows that the series
$$\sum_{m=1}^\infty r_{m_1}q_0^{(k)}\delta(p_m(v_mx_{n(m,k)}v_m^*+v_m^*x_{n(m,k)}v_m)p_m)q_0^{(k)}r_{m_2}$$
also converges with respect to this topology, in addition, we have the following equalities:
\begin{align}
r_{m_1}a_kr_{m_2}+r_{m_2}a_kr_{m_1}=&\sum_{m=1}^\infty r_{m_1}(q_0^{(k)}\delta(w_m)q_0^{(k)})r_{m_2}+r_{m_2}(q_0^{(k)}\delta(w_m)q_0^{(k)})r_{m_1}\notag\\
=&\sum_{m=1}^\infty q_0^{(k)}(r_{m_1}\delta(w_m)r_{m_2}+r_{m_2}\delta(w_m)r_{m_1})q_0^{(k)}\notag\\
=&\sum_{m=1}^\infty q_0^{(k)}(\delta(r_{m_1}w_mr_{m_2})-\delta(r_{m_1})w_mr_{m_2}-\delta(r_{m_2})w_mr_{m_1}\notag\\
&-r_{m_1}w_m\delta(r_{m_2})-r_{m_2}w_m\delta(r_{m_1}))q_0^{(k)}\notag\\
=&q_0^{(k)}(-\delta(r_{m_1})w_{m_2}r_{m_2}-\delta(r_{m_2})w_{m_1}r_{m_1}\notag\\
&-r_{m_1}w_{m_1}\delta(r_{m_2})-r_{m_2}w_{m_2}\delta(r_{m_1}))q_0^{(k)}.                                                    \label{343}
\end{align}
By \eqref{342} and \eqref{343}, we have that
$$r_{m_1}q_0^{(k)}\delta(y_k)q_0^{(k)}r_{m_2}+r_{m_2}q_0^{(k)}\delta(y_k)q_0^{(k)}r_{m_1}=r_{m_1}a_kr_{m_2}+r_{m_2}a_kr_{m_1},$$
that is
$$r_{m_1}(\delta(y_k)-a_k)r_{m_2}=r_{m_2}(a_k-\delta(y_k))r_{m_1}.$$
It follows that
$$r_{m_1}(\delta(y_k)-a_k)r_{m_2}p_{m_2}=r_{m_2}(a_k-\delta(y_k))r_{m_1}p_{m_2}=0.$$
Hence, we can obtain that
\begin{align*}
r_{m_1}(\delta(y_k)-a_k)r_{m}=0
\end{align*}
for every $m\in\mathbb{N}$.
Suppose that $r(r_{m_1}(\delta(y_k)-a_k))$ is the right support of
$r_{m_1}(\delta(y_k)-a_k)$. Hence
$$r(r_{m_1}(\delta(y_k)-a_k))\leq1-r_m$$
for every $m\in\mathbb{N}$.
Therefore by \eqref{330}, we have that
$$r(r_{m_1}(\delta(y_k)-a_k))\leq\mathrm{inf}(1-r_m)=1-q_0^{(k)}.$$
It implies that $r_{m_1}(\delta(y_k)-a_k)q_0^{(k)}=0$ for every $m_1\in\mathbb{N}$.
Similarly, we have that
$q_0^{(k)}(\delta(y_k)-a_k)q_0^{(k)}=0$.
Since $q_0^{(k)}a_kq_0^{(k)}=a_k$, we can obtain that
$$a_k=q_0^{(k)}\delta(y_k)q_0^{(k)}.$$

By \eqref{341} we have that
\begin{align}
\|q_0^{(k)}(\delta(y_k)-y)q_0^{(k)}\|_{\mathcal M}\leq(k+1)^{-1}                   \label{344}
\end{align}
By \eqref{334} and \eqref{344}, we have that
\begin{align}
\|(k+1)y_kq_0^{(k)}\|_{\mathcal M}=\|(k+1)q_0^{(k)}y_k\|_{\mathcal M}\leq1                  \label{345}
\end{align}
and
\begin{align}
\|q_0^{(k)}\delta((k+1)y_k)q_0^{(k)}-(k+1)q_0^{(k)}yq_0^{(k)}\|_{\mathcal M}\leq1.                 \label{346}
\end{align}
By \eqref{345} and \eqref{346}, we know that
$$(k+1)q_0^{(k)}-q_0^{(k)}\delta((k+1)y_k)q_0^{(k)}\leq(k+1)q_0^{(k)}yq_0^{(k)}-q_0^{(k)}\delta((k+1)y_k)q_0^{(k)}\leq q_0^{(k)},$$
that is
\begin{align}
kq_0^{(k)}\leq q_0^{(k)}\delta((k+1)y_k)q_0^{(k)}.                                    \label{347}
\end{align}

Next we let $q_0=\mathrm{inf}_{k\geq1}q_0^{(k)}$ and $z_0=\mathrm{inf}_{k\geq1}z_0^{k}$.
It follows that
$p_0-q_0=\mathrm{sup}_{k\geq1}(p_0-q_0^{(k)})$.
By \eqref{333} and Condition $\mathbb{D}_6$, we can obtain that
$$\Delta(z_0(p_0-q_0))=\Delta(\mathrm{sup}_{k\geq1}z_0(p_0-q_0^{(k)}))\leq\sum_{k=1}^{\infty}\Delta(z_0(p_0-q_0^{(k)}))\leq\varphi(z_0),$$
by Condition $\mathbb{D}_1$, we know that $z_0(p_0-q_0)$ is a finite projection.
Moreover, by \eqref{345} and \eqref{347}, we obtain that
\begin{align}
\|(k+1)q_0y_k\|_{\mathcal M}=\|(k+1)y_k q_0\|_{\mathcal M}\leq1                                      \label{348}
\end{align}
and
\begin{align}
kq_0\leq q_0\delta((k+1)y_k)q_0                                \label{349}
\end{align}
for every $k$ in $\mathbb{N}$.

Since $\varphi$ is a $*$-isomorphism from $\mathcal Z(\mathcal M)$
onto $L^\infty(\Omega,\Sigma,\mu)$, by \eqref{332} we have that
$$\int_\Omega\varphi(1-z_0)d\mu=\int_\Omega \mathrm{sup}_{k\geq1}\varphi(1-z_0^{(k)})d\mu\leq\sum_{k=1}^{\infty}\int_\Omega\varphi(1-z_0^{(k)})d\mu\leq2^{-1},$$
it means that $z_0\neq0$. Since $\mathcal C_{p_0}=1$ and $\mathcal C_{z_0 p_0}= z_0\mathcal C_{p_0}=z_0\neq0$,
we have that $z_0p_0\neq0$, and there exists $n\in\mathbb{N}$ such that $z_0p_n\neq0$.
Since $z_0p_n\sim z_0p_m$, we have that $z_0p_m\neq0$ for every $m\in\mathbb{N}$. Hence, $z_0p_0$
is an infinite projection.
Since  $z_0(p_0-q_0)$ is finite, we know that $z_0q_0$ is a infinite projection.
By \cite[Proposition 6.3.7]{R.Kadison J. Ringrose}, we know that there exists a non-zero
central projection $e_0\leq z_0$ such that $e_0q_0$ is properly infinite,
in particular, there exist pairwise orthogonal projections
\begin{align}
e_n\leq e_0q_0~~\mathrm{and}~~e_n\sim e_0q_0                             \label{350}
\end{align}
for every $n$ in $\mathbb{N}$.
In addition,
\begin{align}
\int_\Omega\varphi(\mathcal C_{q_0}e_0)d\mu\neq0.                         \label{351}
\end{align}

For every $n$ in $\mathbb{N}$, the operator
$$b_n=\delta(e_n)e_n$$
is locally measurable, and there exists a sequence
$\{z_m^{(n)}\}\subset\mathcal P(\mathcal Z(\mathcal M))$
such that
${z_m^{(n)}}\uparrow1$ when $m\rightarrow\infty$
and $z_m^{(n)}b_n\in S(\mathcal M)$ for every $m$ in $\mathbb{N}$.
Since
$$\varphi(z_m^{(n)})\uparrow\varphi(1)=1_{L^\infty(\Omega)},$$
it follows that
$\int_\Omega\varphi(z_m^{(n)})d\mu\uparrow\mu(1_{L^\infty(\Omega)})=1$ when $m\rightarrow\infty$.
For every $n$ in $\mathbb{N}$, by \eqref{351},
there exists a central projection $z^{(n)}$ such that $z^{(n)}b_n\in S(\mathcal M)$ and
\begin{align}
1-2^{-n-1}\int_\Omega\varphi(\mathcal C_{q_0}e_0)d\mu<\int_\Omega\varphi(z^{(n)})d\mu.                       \label{352}
\end{align}
Consider the central projection
$$g_0={\inf_{n\geq1}}z^{(n)}.$$
Since $z^{(n)}b_n\in S(\mathcal M)$ and $g_0=g_0z^{(n)}$, we have that $g_0b_n\in S(\mathcal M)$
for every $n$ in $\mathbb{N}$. By \eqref{352}, we can obtain that
\begin{align*}
1-\int_{\Omega}\varphi(g_0)d\mu=&\int_{\Omega}\varphi(1-g_0)d\mu=\int_{\Omega}\mathrm{sup}\varphi(1-z^{(n)})d\mu\\
\leq&\sum_{n=1}^\infty\int_{\Omega}\varphi(1-z^{(n)})d\mu=\sum_{n=1}^\infty(1-\int_{\Omega}\varphi(z^{(n)})d\mu)\\
\leq&2^{-1}\int_\Omega\varphi(\mathcal C_{q_0}e_0)d\mu.
\end{align*}
It implies that
$$1-2^{-1}\int_\Omega\varphi(\mathcal C_{q_0}e_0)d\mu\leq\int_{\Omega}\varphi(g_0)d\mu.$$
Thus
\begin{align}
1+2^{-1}\int_\Omega\varphi(\mathcal C_{q_0}e_0)d\mu\leq\int_\Omega\varphi(g_0)d\mu+\int_\Omega\varphi(\mathcal C_{q_0}e_0)d\mu.                  \label{353}
\end{align}
By \eqref{351} and \eqref{353},
it follows that $\int_\Omega\varphi(g_0\mathcal C_{q_0}e_0)d\mu>0$,
that is, $g_0\mathcal C_{q_0}e_0\neq0$ and so $g_0e_0q_0\neq0$.
Since $e_0q_0$ is a properly infinite projection, we have that $g_0e_0q_0$ is properly, since $g_0e_n\sim g_0e_0q_0$, we have that $g_0e_n$
is a properly infinite projection for every $n$ in $\mathbb{N}$. Since
$$\mathcal C_{g_0e_n}=g_0C_{e_n}\leq g_0\mathcal C_{q_0e_0}=g_0\mathcal C_{q_0}e_0,$$
it follows that $ze_n$ is also a properly infinite projection for every $0\neq z\in\mathcal P(\mathcal Z(\mathcal M))$
with $z\leq g_0\mathcal C_{q_0}e_0$. In fact, if $z'\in\mathcal P(\mathcal Z(\mathcal M))$
and $z'ze_n\neq0$, then $0\neq z'ze_n=(z'z\mathcal C_{q_0}e_0)g_0e_n$. Since
the projection $g_0e_n$ is properly infinite, we have
$z'z\mathcal C_{q_0}e_0)g_0e_n\notin\mathcal P_{fin}(\mathcal M)$.
Hence, the projection $ze_n$ is also a properly infinite projection.

We may assume that $g_0\mathcal C_{q_0}e_0=1$, otherwise, we replace the algebra
$\mathcal M$ with the algebra $g_0\mathcal C_{q_0}e_0\mathcal M$. In this case,
we also assume that $b_n\in S(\mathcal M)$, $e_n\sim q_0$, $\mathcal C_{e_n}=1$
and $ze_n$ is a properly infinite projection for every non-zero central projection
$z\in\mathcal P(\mathcal Z(\mathcal M))$.

Since $b_n\in S(\mathcal M)$, fixed $n\in\mathbb{N}$,
there exists a sequence $\{p_m^{(n)}\}_{m=1}^\infty\subset\mathcal P_{fin}(\mathcal M)$
such that ${p_m^{(n)}}\downarrow0$ when $m\rightarrow\infty$ and $b_n(1-p_m^{(n)})\in\mathcal M$
for every $m\in\mathbb{\mathrm{N}}$. By Condition $\mathbb{D}_7$, we have that
$\Delta(p_m^{(n)})\in L^0(\Omega,\Sigma,\mu)$ and $\Delta(p_m^{(n)})\downarrow0$,
it follows that $\{\Delta(p_m^{(n)})\}_{n=1}^\infty$ converges in measure $\mu$
to zero.
Hence we can choose a central projection $f_n$ and a finite projection $s_n=p_{m_n}^{(n)}$ as to guarantee
$\Delta(f_ns_n)<2^{-n}\varphi(f_n)$, $1-2^{-n-1}<\int_\Omega\varphi(f_n)d\mu$ and
$f_nb_n(1-s_n)\in\mathcal M$ for every $n\in\mathbb{N}$.

Denote by
$$f={\inf_{n\geq1}}f_n~~\mathrm{and}~~s={\sup_{n\geq1}}s_n.$$
By Condition $\mathbb{D}_6$, we have that
$$1/2<\int_\Omega\varphi(f)d\mu~~\mathrm{and}~~\Delta(fs)\leq\sum_{n=1}^\infty\Delta(fs_n)\leq\varphi(f),$$
it means that $f\neq0$ and by Condition $\mathbb{D}_1$, we know that $fs$ is a finite projection.
Moreover, since $f\leq f_n$ and $(1-s)\leq(1-s_n)$, it follows that
$fb_n(1-s)\in\mathcal M$ for every $n\in\mathbb{N}$.

Let $t=f(1-s)$ and $g_n=f(e_n\wedge(1-s))$ for every $n\in\mathbb{N}$.
By \eqref{350}, we have that
\begin{align}
g_n\leq fe_n\leq q,~~b_ng_n\in\mathcal M~~\mathrm{and}~~g_n\leq t                 \label{354}
\end{align}
for every $n\in\mathbb{N}$.
In addition, we can obtain that
$$fe_n-g_n=f(e_n-e_n\wedge(1-s))\sim f(e_n\vee(1-s)-(1-s))\leq fs,$$
that is $fe_n-g_n$ is a finite projection. Hence, for every non-zero central projection $z\leq f$,
we have that the projection $ze_n-zg_n$ is a finite projection.
Since the projection $ze_n$ is infinite, the projection $zg_n$ is also infinite,
that is
\begin{align}
zg_n\notin\mathcal P_{fin}(\mathcal M)                                                 \label{359}
\end{align}
for every $0\neq z\in\mathcal P(\mathcal Z(\mathcal M))$ and every $n\in\mathbb{N}$.

Since $b_nt=fb_n(1-s)\in\mathcal M$, there exists an increasing sequence $\{l_n\}\subset\mathbb{N}$
such that $l_n>n+2\|b_n\|_{\mathcal M}$ for every $n\in\mathbb{N}$.
By \eqref{348} and \eqref{354}, we have that
\begin{align*}
\|g_n(l_n+1)y_{l_n}\delta(e_n)e_ng_n\|_{\mathcal M}&\leq\|g_n(l_n+1)y_{l_n}\|_{\mathcal M}\|\delta(e_n)e_ng_n\|_{\mathcal M}\\
&\leq\|q_0(l_n+1)y_{l_n}\|_{\mathcal M}\|\delta(e_n)e_nt\|_{\mathcal M}\\
&<(l_n-n)/2.
\end{align*}
Hence,
\begin{align}
\|g_ne_n\delta(e_n)(l_n+1)y_{l_n}g_n+g_n(l_n+1)y_{l_n}\delta(e_n)e_ng_n\|_{\mathcal M}\leq l_n-n.               \label{355}
\end{align}

For every $x=x^*\in\mathcal M$, we have that $-\|x\|_{\mathcal M}1\leq x\leq\|x\|_{\mathcal M}1$,
moreover, $-g_n\|x\|_{\mathcal M}\leq g_nxg_n\leq g_n\|x\|_{\mathcal M}$. By \eqref{355}, it follows that
\begin{align}
g_ne_n\delta(e_n)(l_n+1)y_{l_n}g_n+g_n(l_n+1)y_{l_n}\delta(e_n)e_ng_n\geq(n-l_n)g_n.               \label{356}
\end{align}

Since $e_ne_m=0$ when $n\neq m$. By \eqref{348} and \eqref{354},
we have that the series $\sum_{n=1}^\infty e_n(l_n+1)y_{l_n}e_n$ converges with the topology
$\tau_{so}$ to a self-adjoint operator $h_0\in\mathcal M$ such that
$$\|h_0\|_{\mathcal{M}}\leq{\sup_{n\geq1}}\|e_n(l_n+1)y_{l_n}e_n\|_{\mathcal{M}}\leq1.$$

By \eqref{349}, \eqref{354} and \eqref{356}, we have that
\begin{align*}
ng_n=&l_ng_n+(n-l_n)g_n\\
\leq&g_n(l_n+1)\delta(y_{l_n})g_n+g_ne_n\delta(e_n)(l_n+1)y_{l_n}g_n+g_n(l_n+1)y_{l_n}\delta(e_n)e_ng_n\\
=&(l_n+1)(g_n\delta(y_{l_n})g_n+g_ne_n\delta(e_n)y_{l_n}g_n+g_ny_{l_n}\delta(e_n)e_ng_n)\\
=&(l_n+1)(g_n\delta(y_{l_n})g_n+g_n\delta(e_n)y_{l_n}e_ng_n+g_ne_ny_{l_n}\delta(e_n)g_n)\\
=&(l_n+1)g_n\delta(e_ny_{l_n}e_n)g_n\\
=&(l_n+1)(\delta(g_ne_ny_{l_n}e_ng_n)-\delta(g_n)e_ny_{l_n}e_ng_n-g_ne_ny_{l_n}e_n\delta(g_n))\\
=&\delta(g_nh_0g_n)-\delta(g_n)h_0g_n-g_nh_0\delta(g_n)\\
=&g_n\delta(h_0)g_n.
\end{align*}
Thus,
\begin{align}
ng_n\leq g_n\delta(h_0)g_n                       \label{357}
\end{align}
for every $n$ in $\mathbb{N}$.

Let $g_n^{(0)}=g_n\wedge E_{n-1}(\delta(h_0))$, $n\in\mathbb{N}$,
where $\{E_{\mu}(\delta(h_0))\}$ is the spectral family of projections for self-adjoint operator $\delta(h_0)$.
For every $n$ in $\mathbb{N}$, by \eqref{357} we have that
\begin{align*}
ng_n^{(0)}=&ng_n^{(0)}g_ng_n^{(0)}\leq g_n^{(0)}(g_n\delta(h_0)g_n)g_n^{(0)}\\
=&g_n^{(0)}\delta(h_0)g_n^{(0)}=g_n^{(0)}E_{n-1}(\delta(h_0))\delta(h_0)g_n^{(0)}\\
\leq&g_n^{(0)}(n-1)E_{n-1}(\delta(h_0))g_n^{(0)}=(n-1)g_n^{(0)}.
\end{align*}
Hence, $g_n\wedge E_{n-1}(\delta(h_0))=g_n^{(0)}=0$, it implies that
$$g_n=g_n-g_n\wedge E_{n-1}(\delta(h_0))\sim g_n\vee E_{n-1}(\delta(h_0))-E_{n-1}(\delta(h_0))\leq1-E_{n-1}(\delta(h_0)).$$
Thus
$$g_n\preceq1-E_{n-1}(\delta(h_0)).$$
By \eqref{354}, we have that $g_n\leq fg_n\preceq f(1-E_{n-1}(\delta(h_0)))$,
and by Condition $\mathbb{D}_3$, we can obatain that
\begin{align}
\Delta(g_n)\leq\Delta(f(1-E_{n-1}(\delta(h_0)))).                                                 \label{358}
\end{align}
for every $n$ in $\mathbb{N}$.

Since $|f\delta(h_0)|\in LS(\mathcal{M})$, there exists a non-zero central projection
$f_0\leq f$, such that $|f_0\delta(h_0)|$ is a self-adjoint element in $S(\mathcal{M})$.
By \cite{M. Muratov}, we can
select $\lambda_0>0$ such that
$$(f_0-E_{\lambda}(|f_0\delta(h_0)|))\in\mathcal P_{fin}(\mathcal M)$$
for every $\lambda>\lambda_0$. Thus
$\Delta(f_0(1-E_{\lambda}(|f_0\delta(h_0)|)))\in L_+^0(\Omega,\Sigma,\mu)$ for every $\lambda>\lambda_0$.

Since $f_0(1-E_{\lambda}(|f_0\delta(h_0)|))=f_0(1-E_{\lambda}(|\delta(h_0)|))$,
by \eqref{358} we have that
$$\Delta(f_0g_n)\in L^0_+(\Omega,\Sigma,\mu)$$
for every $n\geq\lambda_0+1$, by Condition $\mathbb{D}_1$, we know that $f_0g_n$
is a finite projection. By \eqref{359},
we know that $f_0g_n$ is an infinite projection,
this leads to a contradiction.

Hence, every Jordan derivation $\delta$ from $LS(\mathcal M)$ into itself is continuous
with respect the local measure topology $t(\mathcal{M})$.
\end{proof}

\section{Extension of a Jordan derivation up to a Jordan derivation on $LS(\mathcal M)$}
In this section, we construct an extension of a Jordan derivation
from a von Neumann algebra $\mathcal M$ into $LS(\mathcal M)$ up to a Jordan
derivation from $LS(\mathcal M)$ into itself.

Let $\mathcal M$ be a von Neumann algebra and $\{z_n\}_{n=1}^{\infty}$ be a sequence
of central projections in $\mathcal M$ such that $z_n\uparrow1$.
A sequence $\{x_n\}_{n=1}^{\infty}$ in $LS(\mathcal M)$ is called \emph{consistent}
with the sequence $\{z_n\}_{n=1}^{\infty}$ if $x_mz_n=x_nz_n$
for each $m,n\in\mathbb{N}$ with $m>n$.

\begin{lemma}\cite[Proposition 4.1]{Ber3}
Suppose that $\mathcal M$ is a von Neumann algebra, $\{x_n\}_{n=1}^{\infty}$,
$\{y_n\}_{n=1}^{\infty}$ are two sequences in $LS(\mathcal M)$,
and $\{z_n\}_{n=1}^{\infty}$, $\{z_n'\}_{n=1}^{\infty}$ are two sequences
of central projections in $\mathcal M$ such that $z_n\uparrow1$
and $z_n'\uparrow1$, respectively.
If $\{x_n\}_{n=1}^{\infty}$ and $\{y_n\}_{n=1}^{\infty}$
consistent with $\{z_n\}_{n=1}^{\infty}$
and $\{z_n'\}_{n=1}^{\infty}$, respectively, then we have the following two statements:\\
$(1)$ there exists a unique element $x\in LS(\mathcal M)$ such that $xz_n=z_nx$ for every $n\in\mathbb{N}$,
moreover, $x_n\xrightarrow[n\rightarrow\infty]{t(\mathcal M)}x$;\\
$(2)$ if $x_nz_nz_m'=y_mz_nz_m'$ for each $m,n\in\mathbb{N}$, then $(x_nz_n-y_nz_n')\xrightarrow[n\rightarrow\infty]{t(\mathcal M)}0$.
\end{lemma}

In the following we extend
a Jordan derivation from $S(\mathcal M)$ into $LS(\mathcal M)$
up to a Jordan derivation from $LS(\mathcal M)$
into itself.
For every $x$ in $LS(\mathcal M)$,
we can select a sequence
$\{z_n\}_{n=1}^{\infty}$ of non-zero central projections
in $\mathcal M$ such that $z_n\uparrow1$ and
$xz_n\in S(\mathcal M)$ for every $n$ in $\mathbb{N}$,
By Lemma 3.2 (2) we know that $\delta(xz_n)z_m=\delta(xz_nz_m)=\delta(xz_nz_n)=\delta(xz_n)z_n$
for each $m,n\in\mathbb{N}$ with $m>n$. It means that the sequence
$\{\delta(xz_n)\}_{n=1}^{\infty}$
is consistent with the sequence $\{z_n\}_{n=1}^{\infty}$. By Proposition 4.1 (1), there exists
a unique element $y_x\in LS(\mathcal M)$ such that $\delta(xz_n)\xrightarrow[n\rightarrow\infty]{t(\mathcal M)}y_x$.
Hence we can define a mapping from $LS(\mathcal M)$ into itself by
$$\tilde{\delta}(x)=y_x.$$
By Proposition 4.1 (2), we know that
$\tilde{\delta}(x)$ does not depend on a choice of a
sequence $\{z_n\}_{n=1}^{\infty}\subset\mathcal P(\mathcal Z(\mathcal M))$.
Thus the mapping $\tilde{\delta}$ is well-defined.
Moreover, if $x\in S(\mathcal M)$, we let $z_n=1$
for every $n\in\mathbb{N}$, then $\tilde{\delta}(x)=\delta(x)$.

\begin{theorem}
Suppose that $\mathcal M$ is a von Neumann algebra. If $\delta$ is a
Jordan derivation from $S(\mathcal M)$ into $LS(\mathcal M)$,
then $\tilde{\delta}$ is a unique Jordan derivation from $LS(\mathcal M)$
into $LS(\mathcal M)$ such that $\tilde{\delta}(x)=\delta(x)$ for every $x\in S(\mathcal M)$.
\end{theorem}

\begin{proof}
First, we prove that $\tilde{\delta}$ is a linear mapping from $LS(\mathcal M)$ into itself.
Suppose that $x$ and $y$ are two elements in $LS(\mathcal M)$, we can select two sequences
$\{z_n\}_{n=1}^{\infty}$ and $\{z_n'\}_{n=1}^{\infty}$ of non-zero central projections
in $\mathcal M$ such that $z_n\uparrow1,z_n'\uparrow1$,
$xz_n,yz_n'\in S(\mathcal M)$ for every $n$ in $\mathbb{N}$, respectively.
It is clear that $\{z_nz_n'\}_{n=1}^{\infty}$ is also a sequence of non-zero central projections
in $\mathcal M$ such that
$$z_nz_n'\uparrow1,~~xz_nz_n',~~yz_nz_n'~~\mathrm{and}~~(x+y)z_nz_n'\in S(\mathcal M)$$
for every $n$ in $\mathbb{N}$.
It follows that
\begin{align*}
\tilde{\delta}(x+y)=&t(\mathcal M)-{\lim_{n \to \infty}\delta((x+y)z_nz_n')}\\
=&(t(\mathcal M)-{\lim_{n \to \infty}\delta(xz_nz_n')})+(t(\mathcal M)-{\lim_{n \to \infty}\delta(yz_nz_n')})\\
=&\tilde{\delta}(x)+\tilde{\delta}(y).
\end{align*}
Similarly, we can show that $\tilde{\delta}(\lambda x)=\lambda\tilde{\delta}(x)$
for every $\lambda\in\mathbb{C}$. Thus $\tilde{\delta}$ is a linear mapping
from $LS(\mathcal M)$ into itself.

Next, we prove that $\tilde{\delta}$ is a Jordan derivation from $LS(\mathcal M)$ into itself.
Since $xz_n\xrightarrow[n\rightarrow\infty]{t(\mathcal M)}x$,
$\delta(xz_n)\xrightarrow[n\rightarrow\infty]{t(\mathcal M)}\tilde{\delta}(x)$,
and $x^2z_n\in S(\mathcal M)$, we have that
\begin{align*}
\tilde{\delta}(x^2)=&t(\mathcal M)-{\lim_{n \to \infty}\delta(x^2z_n)}
=t(\mathcal M)-{\lim_{n \to \infty}\delta(xz_nxz_n)}\\
=&t(\mathcal M)-{\lim_{n \to \infty}(\delta(xz_n)(xz_n)+(xz_n)\delta(xz_n))}\\
=&\tilde{\delta}(x)x+x\tilde{\delta}(x).
\end{align*}
It means that $\tilde{\delta}$ is a Jordan derivation from $LS(\mathcal M)$ into itself.
Clearly, $\tilde{\delta}(x)=\delta(x)$
for every $x$ in $S(\mathcal M)$.

Finally, we show the uniqueness of the $\tilde{\delta}$.
Suppose that $\delta_1$ is also a Jordan derivation from $LS(\mathcal M)$
into itself such that $\delta_1(x)=\delta(x)$ for every $x$ in $S(\mathcal{M})$.

Let $x$ be in $LS(\mathcal M)$, $\{z_n\}_{n=1}^{\infty}$ be a sequence of
non-zero central projections in $\mathcal M$ such that $z_n\uparrow1$,
and $xz_n\in S(\mathcal M)$ for every $n$ in $\mathbb{N}$. By Lemma 3.2 (2) and Proposition
4.1 (1), we have that
\begin{align*}
\tilde{\delta}(x)&=t(\mathcal M)-{\lim_{n \to \infty}\delta(xz_n)}
=t(\mathcal M)-{\lim_{n \to \infty}\delta_1(xz_n)}\\
&=t(\mathcal M)-{\lim_{n \to \infty}\delta_1(x)z_n}=\delta_1(x)
\end{align*}
\end{proof}

Next, we extend a Jordan derivation
from $\mathcal M$ into $LS(\mathcal M)$ up to a Jordan derivation
from $S(\mathcal M)$ into $LS(\mathcal M)$. Let $x$ be in $LS(\mathcal M)$ and
$x=u|x|$ be the polar decomposition of $x$, $l(x)=uu^*$ is the
left support of $x$ and $r(x)=u^*u$ is  the
right support of $x$, clearly, $l(x)\sim u(x)$.
Denote by $s(x)=l(x)\vee r(x)$.

The following lemmas will be used repeated.

\begin{lemma}
Suppose that $\Delta$ is a dimension function of a von Neumann algebra $\mathcal M$.
If $\delta$ is a Jordan derivation from $\mathcal M$ into $LS(\mathcal M)$,
then we have the following inequality
$$\Delta(s(\delta(x)))\leq6\Delta(s(x))$$
for every $x\in\mathcal M$.
\end{lemma}

\begin{proof}
Let $x$ be in $\mathcal M$, we have the following three statements:
\begin{align}
l(\delta(s(x))x)\sim r(\delta(s(x))x)=r(\delta(s(x))xs(x))\leq s(x);         \label{401}
\end{align}
\begin{align}
r(x\delta(s(x)))\sim l(x\delta(s(x)))=l(s(x)x\delta(s(x)))\leq s(x);           \label{402}
\end{align}
and
\begin{align}
r(s(x)\delta(x)s(x))\sim l(s(x)\delta(x)s(x))\leq s(x).                      \label{403}
\end{align}
By \eqref{401}, we have that
$$l(\delta(s(x))x)\preceq s(x)~~\mathrm{and}~~r(\delta(s(x))x)\preceq s(x);$$
by \eqref{402}, we have that
$$l(x\delta(s(x)))\preceq s(x)~~\mathrm{and}~~r(x\delta(s(x)))\preceq s(x);$$
by \eqref{403}, we have that
$$r(s(x)\delta(x)s(x))\preceq s(x).$$

It implies that the following three inequalities:
\begin{align}
\Delta(l(\delta(s(x))x))=\Delta(r(\delta(s(x))x))\leq\Delta(s(x));                     \label{404}
\end{align}
\begin{align}
\Delta(l(x\delta(s(x))))=\Delta(r(x\delta(s(x))))\leq \Delta(s(x));                     \label{405}
\end{align}
and
\begin{align}
\Delta(r(s(x)\delta(x)s(x)))=\Delta(l(s(x)\delta(x)s(x)))\leq \Delta(s(x)).                     \label{406}
\end{align}
Since $\delta$ is a Jordan derivation and by Lemma 3.1 (2) we have that
\begin{align*}
\delta(x)=\delta(s(x)xs(x))=&\delta(s(x))xs(x)+s(x)\delta(x)s(x)+s(x)x\delta(s(x))\\
=&\delta(s(x))x+s(x)\delta(x)s(x)+x\delta(s(x)).
\end{align*}
By \eqref{404}, \eqref{405} and \eqref{406} we can obtain that
\begin{align*}
s(\delta(x))=&s(\delta(s(x))x+s(x)\delta(x)s(x)+x\delta(s(x)))\\
=&r(\delta(s(x)x)\vee l(\delta(s(x)x)\vee r(x\delta(s(x)))\vee l(x\delta(s(x)))\\
&\vee r(s(x)\delta(x)s(x))\vee l(s(x)\delta(x)s(x)).
\end{align*}
It implies that
$$\Delta(s(\delta(x)))\leq6\Delta(s(x)).$$
for every $x$ in $\mathcal M$.
\end{proof}

The following result is proved by Ber, Chilin and Sukochev in \cite{Ber3}.

\begin{lemma}\cite[Lemma 4.4]{Ber3}
Suppose that $\{x_n\}_{n=1}^\infty$ is a sequence in $LS(\mathcal M)$. If $s(x_n)\in\mathcal P_{fin}(\mathcal M)$
and $\Delta(s(x_n))\xrightarrow[n\rightarrow\infty]{t(L^\infty(\Omega))}0$, then $x_n\xrightarrow[n\rightarrow\infty]{t(\mathcal M)}0$.
In particular, if $\{p_n\}_{n=1}^{\infty}\subset\mathcal P_{fin}(\mathcal M)$ and $p_n\downarrow0$,
then $p_n\xrightarrow[n\rightarrow\infty]{t(\mathcal M)}0$.
\end{lemma}

By Lemmas 4.3 and 4.4, we have the following result.

\begin{lemma}
Suppose that $\mathcal M$ is a von Neumann algebra.
Let $x$ be in $S(\mathcal M)$, $\{p_n\}_{n=1}^{\infty},\{q_n\}_{n=1}^{\infty}$
be two sequences of non-zero projections
in $\mathcal M$ such that $p_n\uparrow1,q_n\uparrow1$,
$xp_n,xq_n\in\mathcal M$ and $p_n^\bot,q_n^\bot\in\mathcal P_{fin}(\mathcal M)$
for every $n$ in $\mathbb{N}$.
If $\delta$ is a Jordan derivation from $\mathcal M$ into $LS(\mathcal M)$,
then there exists an element $\hat{\delta}(x)$
in $LS(\mathcal M)$ such that
$$t(\mathcal M)-{\lim_{n \to \infty} \delta(xp_n)}=\hat{\delta}(x)=t(\mathcal M)-{\lim_{n\to\infty}\delta(xq_n)}.$$
\end{lemma}

\begin{proof}
Let $m$ and $n$ be in $\mathbb{N}$ with $m>n$, we have that
$$l(x(p_m-p_n))\sim r(x(p_m-p_n))\leq p_m-p_n.$$
By Lemma 4.3, Conditions $\mathbb{D}_2$ and $\mathbb{D}_3$, we can obtain that
\begin{align}
\Delta(s(\delta(xp_m-xp_n)))&=\Delta(s(\delta(x(p_m-xp_n))))\leq6\Delta(s(x(p_m-p_n)))\notag\\
&\leq6\Delta(l(x(p_m-p_n))\vee(p_m-p_n))\leq12\Delta(p_m-p_n)\notag\\
&\leq12\Delta(p_n^\bot).                                              \label{407}
\end{align}
By Condition $\mathbb{D}_1$, we know that $\Delta(p_n^\bot)\in L^0_+(\Omega,\Sigma,\mu)$,
and by Condition $\mathbb{D}_7$, we know that $\Delta(p_n^\bot)\downarrow0$,
it follows that $\Delta(p_n^\bot)\xrightarrow[n\rightarrow\infty]{t(L^\infty(\Omega))}0$.
By \eqref{407} we can obtain that
$$\Delta(s(\delta(xp_m)-\delta(xp_n)))\xrightarrow[m,n\rightarrow\infty]{t(L^\infty(\Omega))}0.$$
By Lemma 4.4, we have that
$$(\delta(xp_m)-\delta(xp_n))\xrightarrow[m,n\rightarrow\infty]{t(L^\infty(\Omega))}0.$$
It means that $\{\delta(xp_n)\}_{n=1}^\infty$ is a Cauchy sequence in $(LS(\mathcal M),t(\mathcal M))$.
Hence, there exists an element $\hat{\delta}(x)$ in $LS(\mathcal M)$ such that
$$t(\mathcal M)-{\lim_{n \to \infty} \delta(xp_n)}=\hat{\delta}(x).$$

Next we show that
$$t(\mathcal M)-{\lim_{n \to \infty} \delta(xq_n)}=\hat{\delta}(x).$$

For every $n\in\mathbb{N}$, we have that
\begin{align*}
&(p_n-q_n)((p_n-p_n\wedge q_n)\vee(q_n-p_n\wedge q_n))\\
=&((p_n-p_n\wedge q_n)-(q_n-p_n\wedge q_n))((p_n-p_n\wedge q_n)\vee (q_n-p_n\wedge q_n))\\
=&(p_n-p_n\wedge q_n)-(q_n-p_n\wedge q_n)=p_n-q_n.
\end{align*}
It follows that
$$r(p_n-q_n)=r((p_n-q_n)((p_n-p_n\wedge q_n)\vee(q_n-p_n\wedge q_n)))\leq((p_n-p_n\wedge q_n)\vee(q_n-p_n\wedge q_n)).$$
Since
$$r(x(p_n-q_n))\leq r(p_n-q_n)$$
and
$$l(x(p_n-q_n))\sim r(x(p_n-q_n)).$$
By Condition $\mathbb{D}_6$ we can obtain that
\begin{align}
\Delta(s(x(p_n-q_n)))=&\Delta(l(x(p_n-q_n))\vee r(x(p_n-q_n)))\notag\\
\leq&\Delta(l(x(p_n-q_n)))+\Delta(r(x(p_n-q_n)))\notag\\
=&2\Delta(r(x(p_n-q_n)))\notag\\
\leq&2\Delta((p_n-p_n\wedge q_n)\vee(q_n-p_n\wedge q_n))\notag\\
\leq&2\Delta(p_n-p_n\wedge q_n)+2\Delta(q_n-p_n\wedge q_n)\notag\\
\leq&4\Delta(1-p_n\wedge q_n)\notag\\
=&4\Delta(p_n^\bot\vee q_n^\bot)
\leq4(\Delta(p_n^\bot)+\Delta(q_n^\bot)).                              \label{508}
\end{align}
By Lemma 4.3, we have that
\begin{align}
\Delta(s(\delta(xp_n)-\delta(xq_n)))=\Delta(s(\delta(x(p_n-q_n))))\leq6\Delta(s(x(p_n-q_n))),                             \label{509}
\end{align}
By \eqref{508} and \eqref{509}, we can obtain that
$$\Delta(s(\delta(xp_n)-\delta(xq_n)))\leq24(\Delta(p_n^\bot)+D(q_n^\bot))\downarrow0$$
By Lemma 4.4, we obtain that
$$t(\mathcal M)-{\lim_{n \to \infty} \delta(xq_n)}=\hat{\delta}(x)=t(\mathcal M)-{\lim_{n \to \infty} \delta(xp_n)}$$
for every $x$ in $S(\mathcal M)$.
\end{proof}

By Lemma 4.5, we can extend every Jordan derivation $\delta$ from $\mathcal M$
into $LS(\mathcal M)$ up to a Jordan derivation $\hat{\delta}$ from $S(\mathcal M)$
into $LS(\mathcal M)$.

For every $x$ in $S(\mathcal M)$, there exists a sequence $\{p_n\}_{n=1}^{\infty}$
of non-zero projections in $\mathcal M$ such that $p_n\uparrow1$,
$xp_n\in\mathcal M$ and $p_n^\bot\in\mathcal P_{fin}(\mathcal M)$
for every $n$ in $\mathbb{N}$. By Lemma 4.5, there exists
an element $\hat{\delta}(x)$ in $LS(\mathcal M)$ such that
$$\hat{\delta}(x)=t(\mathcal M)-{\lim_{n \to \infty} \delta(xp_n)}.$$
Moreover, we know that the definition of $\hat{\delta}(x)$ dose not
depend on a choice of sequence $\{p_n\}_{n=1}^{\infty}$, which satisfies
the above properties, in particular, $\hat{\delta}(x)=\delta(x)$
for every $x$ in $\mathcal M$.

\begin{theorem}
Suppose that $\mathcal M$ is a von Neumann algebra and $\delta$ is
a Jordan derivation from $\mathcal M$ into $LS(\mathcal M)$.
Then $\hat{\delta}$ is a unique Jordan derivation
from $S(\mathcal M)$ into $LS(\mathcal M)$ such that $\hat{\delta}(x)=\delta(x)$
for every $x\in\mathcal M$.
\end{theorem}

\begin{proof}
First, we show that $\hat{\delta}$ is a linear mapping from $S(\mathcal M)$ into $LS(\mathcal M)$.
For each $x$ and $y$ in $S(\mathcal M)$, we can select two sequences
$\{p_n\}_{n=1}^{\infty},\{q_n\}_{n=1}^{\infty}\subset\mathcal P(\mathcal M)$,
such that
$$p_n\uparrow1,q_n\uparrow1,xp_n,yq_n\in\mathcal M~~\mathrm{and}~~p_n^\bot,q_n^\bot\in\mathcal P_{fin}(\mathcal M)$$
for every $n\in\mathbb{N}$. Denote by
$$e_n=p_n\wedge q_n.$$
It is easy to show that $\{e_n\}_{n=1}^{\infty}$ is an increase sequence, and we have that
$$xe_n=xp_ne_n\in\mathcal M, ye_n=yq_ne_n\in\mathcal M,$$
and
$$e_n^\bot=p_n^\bot\vee q_n^\bot\in\mathcal P_{fin}(\mathcal M),
\Delta(e_n^\bot)\leq(\Delta(p_n^\bot)+\Delta(q_n^\bot))\downarrow0.$$
By Condition $\mathbb{D}_7$, we can obtain that
$$e_n^\bot\downarrow0~~\mathrm{and}~~e_n\uparrow1.$$
By the definition of $\hat{\delta}$, we have that
\begin{align*}
\hat{\delta}(x+y)=&t(\mathcal M)-{\lim_{n \to \infty}\delta((x+y)e_n)}\\
=&(t(\mathcal M)-{\lim_{n \to \infty}\delta(xe_n)})+(t(\mathcal M)-{\lim_{n \to \infty}\delta(ye_n)})\\
=&\hat{\delta}(x)+\hat{\delta}(y).
\end{align*}
Similarly, we can show that $\hat{\delta}(\lambda x)=\lambda\hat{\delta}(x)$ for every $\lambda\in\mathbb{C}$.

Next we show that $\hat{\delta}$ is a Jordan derivation from $S(\mathcal M)$ into $LS(\mathcal M)$.
By the polar decomposition $x=u|x|$, $u^*u=r(x)$, we have that $x_n=xE_n(|x|)\in\mathcal M$
for every $n\in\mathbb{N}$. Denote by
$$g_n=1-r(E_n^\bot(|x|)x_n)~~\mathrm{and}~~s_n=g_n\wedge E_n(|x|).$$
It implies that
\begin{align}
x_ng_n=E_n(|x|)x_ng_n+E_n^\bot(|x|)x_ng_n=E_n(|x|)x_ng_n.                               \label{408}
\end{align}
By \eqref{408}, we have that
\begin{align}
x^2s_n&=x^2E_{n}(|x|)s_n=xx_ns_n=xx_ng_ns_n\notag\\
&=xE_{n}(|x|)x_ng_ns_n=xE_{n}(|x|)xE_{n}(|x|)s_n.                              \label{501}
\end{align}

Since $g_n^\bot=r(E_n^\bot(|x|)x_n)\sim l(E_n^\bot(|x|)x_n)\leq E_n^\bot(|x|)$,
we obtain that $g_n^\bot\preceq E_n^\bot(|x|)$.
Since $x\in S(\mathcal M)$, there exists $n_0\in\mathbb{N}$ such that $E_n^\bot(|x|)\in\mathcal P_{fin}(\mathcal M)$
for every $n\geq n_0$, it follows that $g_n^\bot\in\mathcal P_{fin}(\mathcal M)$ for every $n\geq n_0$.
Hence we have that
$$s_n^\bot=g_n^\bot\vee E_{n}^\bot(|y|)\in\mathcal P_{fin}(\mathcal M)$$
for every $n>n_0$, and
$$D(s_n^\bot)\leq D(g_n^\bot)+D(E_{n}^\bot(|y|))\leq(D(E_{n}^\bot(|x|))+D(E_{n}^\bot(|y|)))\downarrow0.$$
It means that $s_n^\bot\downarrow0$ and $s_n\uparrow1$.
By \eqref{408}, we have that
\begin{align*}
E_{n+1}^\bot(|x|)x_{n+1}s_n=&E_{n+1}^\bot(|x|)E_{n}^\bot(|x|x_{n+1}E_{n}(|x|))s_n\\
=&E_{n+1}^\bot(|x|)(E_{n}^\bot(|x|)x_{n}E_{n}(|x|))s_n\\
=&E_{n+1}^\bot(|x|)(E_{n}^\bot(|x|)x_{n}s_n)\\
=&E_{n+1}^\bot(|x|)(E_{n}^\bot(|x|)x_{n}g_n)s_n\\
=&0.
\end{align*}
It follows that
$$s_n\leq1-r(E_{n+1}^\bot(|x|)x_{n+1})=g_{n+1}$$
for every $n\in\mathbb{N}$.
By the inequalities $s_n\leq E_{n}(|y|)\leq E_{n+1}(|y|)$,
we have that $s_n\leq s_{n+1}$.

By \eqref{501}, Lemmas 4.4 and 4.5, we can obtain that
\begin{align*}
&\hat{\delta}(x^2)=t(\mathcal M)-{\lim_{n \to \infty}\delta(x^2s_n)}=t(\mathcal M)-{\lim_{n \to \infty}}s_n\delta(x^2s_n)s_n\\
=&t(\mathcal M)-{\lim_{n \to \infty}}[\delta(s_nx^2s_n)-\delta(s_n)x^2s_n-s_nx^2s_n\delta(s_n)]\\
=&t(\mathcal M)-{\lim_{n \to \infty}}[\delta(s_nxE_n(|x|)xE_n(|x|)s_n)-\delta(s_n)x^2s_n-s_nx^2s_n\delta(s_n)]\\
=&t(\mathcal M)-{\lim_{n \to \infty}}[s_n\delta(xE_n(|x|)xE_n(|x|))s_n+\delta(s_n)xE_n(|x|)xE_n(|x|)s_n\\
&+s_nxE_n(|x|)xE_n(|x|)\delta(s_n)-\delta(s_n)x^2s_n-s_nx^2s_n\delta(s_n)]\\
=&t(\mathcal M)-{\lim_{n \to \infty}}[s_n(\delta(xE_n(|x|))xE_n(|x|)+xE_n(|x|)\delta(xE_n(|x|))s_n]\\
=&t(\mathcal M)-{\lim_{n \to \infty}[s_n(\hat{\delta}(x)x+x\hat{\delta}(x))s_n]}\\
=&\hat{\delta}(x)x+x\hat{\delta}(x).
\end{align*}
It means that $\hat{\delta}$ is a Jordan derivation from $S(\mathcal M)$
into $LS(\mathcal M)$. Moreover, we have that
$\hat{\delta}(x)=\delta(x)$ for every $x\in\mathcal M$.

Finally, we show the uniqueness of $\hat{\delta}$.
Suppose that $\delta_1$ is a Jordan derivation from $S(\mathcal M)$
into $LS(\mathcal M)$ such that $\delta_1(y)=\delta(y)$
for every $y\in\mathcal M$.

Let $y$ be in $\mathcal M$, then $E_{n}(|y|)\uparrow1$, $yE_{n}(|y|)\in\mathcal M$,
for every $n\in\mathbb{N}$. There exists $n_1\in\mathbb{N}$ such that $E_{n}^\bot(|y|)\in\mathcal P_{fin}(\mathcal M)$
for every $n\geq n_1$. By Lemma 4.4, we have that
$E_{n}(|y|)\xrightarrow[n\rightarrow\infty]{t(\mathcal M)}1$. Since $(LS(\mathcal M),t(\mathcal M))$
is a topological algebra, it follows that
\begin{align}
\delta_1(y)=&t(\mathcal M)-{\lim_{n \to \infty}}E_{n}(|y|)\delta_1(y)E_{n}(|y|)\notag\\
=&t(\mathcal M)-{\lim_{n \to \infty}}[\delta_1(E_{n}(|y|)yE_{n}(|y|))-\delta_1(E_{n}(|y|))yE_{n}(|y|)\notag\\
&-E_{n}(|y|)y\delta_1(E_{n}(|y|))]\notag\\
=&t(\mathcal M)-{\lim_{n \to \infty}}[\delta_1(E_{n}(|y|)yE_{n}(|y|)E_{n}(|y|))-\delta_1(E_{n}(|y|))yE_{n}(|y|)\notag\\
&-E_{n}(|y|)y\delta_1(E_{n}(|y|))]\notag\\
=&t(\mathcal M)-{\lim_{n \to \infty}}[E_{n}(|y|)\delta_1(yE_{n}(|y|))E_{n}(|y|)+\delta_1(E_{n}(|y|))yE_{n}(|y|)\notag\\
&+E_{n}(|y|)yE_{n}(|y|)\delta_1(E_{n}(|y|))-\delta_1(E_{n}(|y|))yE_{n}(|y|)-E_{n}(|y|)y\delta_1(E_{n}(|y|))]\notag\\
=&\hat{\delta}(y)+t(\mathcal M)-{\lim_{n \to \infty}}[E_{n}(|y|)yE_{n}(|y|)\delta_1(E_{n}(|y|))\notag\\
&-E_{n}(|y|)y\delta_1(E_{n}(|y|))].                                                                              \label{412}
\end{align}
Since $\delta(1)=0$, $s(y)=s(-y)$ for every $y\in LS(\mathcal M)$, by Lemma 4.3, it follows that
\begin{align*}
\Delta(s((\delta(E_n(|y|))))=\Delta(s((\delta(-E_n(|y|))))=\Delta(s((\delta(1-E_n(|y|))))\leq6\Delta(E_n^\bot(|y|))\downarrow0.
\end{align*}
By Lemma 4.4, we obtain $\delta(E_n(|y|))\xrightarrow[n\rightarrow\infty]{t(\mathcal M)}0$,
by \eqref{412}, we have that $\delta_1(y)=\hat{\delta}(y)$.
\end{proof}

\begin{theorem}
Suppose that $\mathcal{M}$ is a von Neumann algebra
and $\mathcal A$ is a subalgebra of $LS(\mathcal M)$ such that
$\mathcal M\subset\mathcal A$.
If $\delta$ is a Jordan derivation from $\mathcal A$ into $LS(\mathcal M)$,
then there exists a unique Jordan derivation $\delta_{\mathcal A}$ from $LS(\mathcal M)$
into $LS(\mathcal M)$ such that $\delta_{\mathcal A}(x)=\delta(x)$ for every $x\in\mathcal A$.
\end{theorem}

\begin{proof}
Since $\mathcal M\subset\mathcal A$, the restriction $\delta_0$ of the Jordan derivation
$\delta$ on $\mathcal M$ is a well-defined Jordan derivation from $\mathcal M$
into $LS(\mathcal M)$. By Theorems 4.2 and 4.7, we know that
$\delta_{\mathcal A}=\tilde{\hat{\delta}}$
is a unique Jordan derivation from $LS(\mathcal M)$ into itself
such that
$\delta_{\mathcal A}(x)=\delta_0(x)$
for every $x$ in $\mathcal M$.

Next we show that
$$\delta_{\mathcal A}(a)=\delta(a)$$
for every $a$ in $\mathcal A$.
Let $a$ be in $\mathcal A$, there exists a
sequence $\{z_n\}_{n=1}^\infty\subset\mathcal P(\mathcal Z(\mathcal M))$,
such that $z_n\uparrow1$ and $az_n\in S(\mathcal M)$ for every $n\in\mathbb{N}$.
By Lemma 4.1 (1), we have that
$z_n\xrightarrow[n\rightarrow\infty]{t(\mathcal M)}1$. By Lemma 3.2 (2), we can obtain that
$$\delta_{\mathcal A}(a)=t(\mathcal M)-{\lim_{n \to \infty}\delta_{\mathcal A}(a)z_n}=t(\mathcal M)-{\lim_{n \to \infty}\delta_{\mathcal A}(az_n)}.$$
Similarly, we have that $\delta(a)=t(\mathcal M)-{\lim_{n \to \infty}\delta(az_n)}$.

Since $\delta_{\mathcal A}(x)=\delta_0(x)=\delta(x)$ for every $x$ in $\mathcal M$
and by the proof of uniqueness of the derivation $\hat{\delta}$ from Proposition 4.7,
we can obtain that $\delta_{\mathcal A}(az_n)=\delta(az_n)$ for every $n$ in $\mathbb{N}$,
it implies that $\delta_{\mathcal A}(a)=\delta(a)$ for every $a$ in $\mathcal A$.
\end{proof}

By Theorems 3.4 and 4.7, we have the following corollary.

\begin{corollary}
Suppose that $\mathcal{M}$ is a properly infinite von Neumann algebra,
$\mathcal A$ is a subalgebra of $LS(\mathcal M)$ such that $\mathcal M\subset\mathcal A$.
If $\delta$ is a Jordan derivation from $\mathcal A$ into $LS(\mathcal M)$, then $\delta$
is continuous with respect to the local measure topology $t(\mathcal M)$.
\end{corollary}

\bibliographystyle{amsplain}

\begin{thebibliography}{99}

\bibitem{Albeverio1} S. Albeverio, S. Ayupov, K. Kudaybergenov.
Derivations on the algebra of measurable operators affiliated with a type I von Neumann algebra.
Siberian Adv. Math., 2008, 18: 86-94.

\bibitem{Albeverio2} S. Albeverio, S. Ayupov, K. Kudaybergenov.
Structure of derivations on various algebras of measurable operators for type I von Neumann algebras.
J. Func. Anal., 2009, 256: 2917-2943.

\bibitem{Alizadeh} R. Alizadeh. Jordan derivations of full matrix algebras.
Linear Algebra Appl., 2009, 430: 574-578.

\bibitem{Ber1} A. Ber, V. Chilin, F. Sukochev.
Non-trivial derivation on commutative regular algebras.
Extracta Math., 2006, 21: 107-147.

\bibitem{Ber3} A. Ber, V. Chilin, F. Sukochev.
Continuity of derivations of algebras of locally measurable operators.
Integr. Equ. Oper. Theory, 2013, 75: 527-557.

\bibitem{Ber2} A. Ber, V. Chilin, F. Sukochev.
Continuous derivations on algebras of locally measurable operators are inner.
Proc. London Math. Soc., 2014, 109: 65-89.

\bibitem{An} G. An, Y. Ding, J. Li. Characterizations of Jordan left derivations on some algebras.
Banach J. Math. Anal., 2016, 10: 466-481.

\bibitem{M. Bresar2} M. Bre\v{s}ar. Jordan derivations on semiprime rings.
Bull. Aust. Math. Soc., 1988, 104: 1003-1006.

\bibitem{M. Bresar J. Vukman1} M. Bre\v{s}ar, J. Vukman. Jordan derivations on prime rings.
Bull. Aust. Math. Soc., 1988, 37: 321-322.

\bibitem{M. Bresar J. Vukman} M. Bre\v{s}ar, J. Vukman.
On left derivations and related mappings. Proc. Amer. Math. Soc., 1990, 11:, 7-16.

\bibitem{Cuntz} J. Cuntz. On the continuity of Semi-Norms on operator algebras.
Math. Ann., 1976, 220: 171-183.

\bibitem{Cusack} J. Cusack. Jordan derivations on rings.
Proc. Amer. Math. Soc., 1975, 53: 321-324.

\bibitem{Q. Deng} Q. Deng. On Jordan left derivations. Math. J. Okayama Univ., 1992, 34: 145-147.

\bibitem{hejazian} S. Hejazian, A. Niknam. Modules, annihilators and module derivations of $JB^{*}$-algebras.
Indian J. Pure Appl. Math., 1996, 27: 129-140.

\bibitem{I. Herstein} I. Herstein, Jordan derivations of prime rings.
Proc. Amer. Math. Soc., 1957, 8: 1104-1110.

\bibitem{Johnson} B. Johnson. Symmetric amenability and the nonexistence of Lie and Jordan derivations.
Math. Proc. Cambd. Philos. Soc., 1996, 120: 455-473.

\bibitem{M. Muratov} M. Muratov, V. Chilin. Algebras of measurable and locally measurable operators.
Kyiv, Pratse In-ty matematiki NAN ukraini., 2007, 69, 390 pp, (Russian).

\bibitem{R.Kadison J. Ringrose} R. Kadison, J. Ringrose, Fundamentals of the Theory of Operator Algebras,
I, Pure Appl. Math. 100, Academic Press, New York, 1983.

\bibitem{Jian kui Li Jiren Zhou} J. Li, J. Zhou.
Jordan left derivations and some left derivable maps. Oper. Matrices, 2010, 4: 127-138.

\bibitem{Ringrose} J. Ringrose. Automatically continuous of derivations of operator algebras. J. London Math. Soc., 1972, 5: 432-438.

\bibitem{sakai} S. Sakai. Derivations of $W^{*}$-algebras. Ann. Math., 1966, 83: 273--279.

\bibitem{Segal} I. Segal. A non-commutative extension of abstract integration. Ann, Math., 1953, 57: 401-457.

\bibitem{M. Takesaki} M. Takesaki. Theory of operator algebras I, New York, Springer-Verlag, 1979.

\bibitem{J. Vukman1} J. Vukman. On left Jordan derivations of rings and Banach algebras. Aequations Math., 2008, 75: 260-266.

\bibitem{F. Yeadon} F. Yeadon. Convergence of measurable operators. Proc. Camb. Phil. Soc,, 1973, 74: 257-268.
\end{thebibliography}

\end{document}